\documentclass[11pt, a4paper]{amsproc}
\usepackage[T1]{fontenc}
\usepackage{lmodern}
\usepackage[a4paper,margin=25mm,footskip=10mm]{geometry}
\usepackage{amssymb,mathtools}
\usepackage{microtype}
\usepackage{xcolor}
\usepackage{enumitem}
\usepackage{needspace,etoolbox}

\usepackage{hyperref}
\usepackage{bookmark}

\numberwithin{equation}{section}
\newtheorem{theorem}{Theorem}[section]
\newtheorem{lemma}[theorem]{Lemma}
\newtheorem{proposition}[theorem]{Proposition}
\newtheorem{corollary}[theorem]{Corollary}
\theoremstyle{definition}
\newtheorem{definition}[theorem]{Definition}
\theoremstyle{remark}

\newtheorem{maincorrection}{Theorem}

\newtheorem{maintheorem}[maincorrection]{Theorem}
\newtheorem{maincorollary}[maincorrection]{Corollary}

\newcommand{\M}{\mathcal M}

\DeclareMathOperator{\Tr}{Tr}

\DeclareMathOperator{\Ad}{Ad}

\DeclareMathOperator{\rank}{rank}
\DeclareMathOperator{\supp}{supp}

\title{Spectral gap, internality of relative commutants, and non-hyperlinear groups}
\author{Vadim Alekseev}
\address{V.A., Institut f\"ur Geometrie, TU Dresden,
01062 Dresden, Germany}
\email{vadim.alekseev@tu-dresden.de}

\author{Jihao Liu}
\address{J.L., Department of Mathematics and Beijing International Center
for Mathematical Research, Peking University, No.~5 Yiheyuan Road,
Haidian District, Beijing 100871, China}
\email{liujihao@math.pku.edu.cn}

\author{Andreas Thom}
\address{A.T., Institut f\"ur Geometrie, TU Dresden,
01062 Dresden, Germany}
\email{andreas.thom@tu-dresden.de}
\date{29 September 2026}

\begin{document}
\begin{abstract}
We prove that the relative commutant of a finite unitary tuple in a
tracial matrix ultraproduct is internal whenever its conjugation operator
has a spectral gap off the commutant. This implies the existence of a non-hyperlinear group.

\end{abstract}
\maketitle\thispagestyle{plain}
\tableofcontents

\section{Introduction}
\sectionmark{Introduction}
\label{intro:section}

In 1976, Connes~\cite[p.~105]{Connes1976} raised the question
whether every factor of type $\mathrm{II}_1$ with separable predual
embeds into an ultrapower $R^\omega$ of the hyperfinite
$\mathrm{II}_1$ factor $R$. Equivalently, the Connes embedding
problem asks whether every tracial von Neumann algebra with
separable predual admits a trace-preserving embedding into
$R^\omega$, or whether every finite tuple in such an algebra
admits finite-dimensional matrix approximations to its joint
$*$-moments. Kirchberg~\cite{Kirchberg1993} established equivalent
formulations in terms of tensor products of $C^*$-algebras and
the QWEP conjecture; see also~\cite{Ozawa2004}.

R\u{a}dulescu~\cite{Radulescu2008} developed the connection with
finite-dimensional approximation of discrete groups and exhibited
a non-residually finite group whose von Neumann algebra embeds into
$R^\omega$. A countable discrete group $\Gamma$ is called
\emph{hyperlinear} if it embeds into the unitary group of a tracial
matrix ultraproduct. This is equivalent to the existence of a
trace-preserving embedding of its group von Neumann algebra
$L(\Gamma)$ into $R^\omega$; see~\cite[Proposition~7.1]{Ozawa2004}.
Every sofic group is hyperlinear, since permutation matrices
convert approximations in normalized Hamming distance into
approximations in normalized Hilbert--Schmidt norm.

Ji, Natarajan, Vidick, Wright, and Yuen~\cite{MIPRE} gave a
negative answer to the Connes embedding problem through their
theorem $\mathrm{MIP}^*=\mathrm{RE}$ in quantum complexity theory.
As they explain in~\cite[Section~1.4]{MIPRE}, the existence of
non-hyperlinear groups remained unresolved by this result.
Subsequent developments extended these methods to other
approximation problems. Bowen, Chapman, Lubotzky, and
Vidick~\cite{AldousLyonsI}, together with the companion work of
Bowen, Chapman, and Vidick~\cite{AldousLyonsII}, disproved the
Aldous--Lyons conjecture by constructing invariant random
subgroups of free groups that are not co-sofic.
Building on their refinement of the nonlocal-game machinery,
Manzoor~\cite{Manzoor} constructed a non-co-hyperlinear invariant
random subgroup and a countable probability-measure-preserving
equivalence relation whose von Neumann algebra is not Connes
embeddable. Taller and Vidick~\cite{TallerVidick} proved
RE-hardness for approximation of the quantum value of linear
constraint system games. Their result does not establish the
perfect completeness needed for the implication to
non-hyperlinear groups discussed there. More recently,
Lin~\cite{LinMIPco} proved
$\mathrm{MIP}^{\mathrm{co}}=\mathrm{coRE}$ for the commuting
operator model, obtaining another negative solution of the
Connes embedding problem.

The approach developed here originates in the rigidity of
approximations of Kazhdan groups. Kun~\cite{KunExpanders} proved
that sofic approximations of such groups can be modified on a
negligible set of edges into disjoint unions of uniformly
expanding graphs. Kun and the third author~\cite{KunThomActions} established
rigidity and self-improvement properties of their almost
automorphisms. These results were used in OpenAI's construction
of a non-sofic group~\cite[Chapter~3]{OpenAINonsofic}, where an
expander-matching argument detects an obstruction arising from
conjugating a Kazhdan subgroup properly into itself.
Kun and the third author~\cite{KunThomNonsofic} subsequently developed this
mechanism into constructions of non-sofic generalized wreath
products and amalgamated doubles.

The first and third author~\cite{AlekseevThomCentralizers}
proved an internality theorem for centralizers of sofic
approximations of Kazhdan groups and asked for an analogous
statement in tracial matrix ultraproducts, see \cite[Open Problem~6.2(a)]{AlekseevThomCentralizers}. Finally, the third author showed that  a positive answer to this so-called centralizer problem would
yield non-hyperlinear amalgamated doubles, see~\cite{ThomConditional}.

To state the problem, let
\[
 \M=\prod_\omega(M_{d_n}(\mathbb C),\tau)
\]
be a tracial matrix ultraproduct, where each matrix algebra
carries its normalized trace. A von Neumann subalgebra of $\M$
is called \emph{internal} if it has the form $[A_n]_\omega$
for unital $*$-subalgebras $A_n\subset M_{d_n}(\mathbb C)$.
The centralizer problem asks whether
$
 \pi(\Gamma)'\cap\M
$
is internal for every Kazhdan group $\Gamma$ and every
homomorphism $\pi:\Gamma\to U(\M)$.

Our main result gives a positive answer to this problem.
In fact, we prove a somewhat stronger statement. Our standing assumption is that we are given a unitary tuple $(u_1,\dots,u_h)$ in $\M$ which has \emph{spectral gap} $\kappa\in(0,1)$, meaning that
\begin{equation}\label{intro:main-gap}
 \frac1{4h}\sum_{j=1}^h\|[u_j,x]\|_2^2
 \ge \kappa\|x-E_C(x)\|_2^2
 \qquad(x\in\M),
\end{equation}
where $E_C$ is the trace-preserving conditional expectation
onto the relative commutant $C = \{u_1,\ldots,u_h\}'\cap\M$.

Our first result is the following correction theorem. It should be regarded as an analogue of the decomposition of a sofic approximation into expanders under a spectral gap assumption \cite[Theorem~1]{KunExpanders} and \cite[Theorem~C]{AlekseevDrigalla}. In the matrix setting, we obtain a decomposition into quantum expanders, see \cite{HastingsExpanders,PisierExpanders,LiQiaoWigderson} for background.

\begin{maincorrection}[Theorem~\ref{phy:decomposition}]
\label{main:correction}
 Let $u_1,\ldots,u_h\in U(\M)$ be a unitary tuple with spectral gap $\kappa \in (0,1)$ in the ultraproduct. There exist lifts $(u_{j,n}^{\pm})$ of $u_j$ to $U(d_n)$ and invariant projection partitions $(e_{0,n},\ldots,e_{m_n,n})$ such that the compressions of the doubled tuple $(u^{\pm}_{j,n})_{j,\pm}$ to each nonzero block $e_{i,n}M_{d_n}e_{i,n}$ have spectral gap $\kappa^2/2^{28}$ relative to the scalars.
\end{maincorrection}

A crucial consequence of this spectral gap lifting theorem is that the whole commutant is in fact internal.

\begin{maintheorem}
\label{main:internality}
Let $u_1,\ldots,u_h\in U(\M)$ be a unitary tuple with spectral gap in the ultraproduct. Then the relative commutant $C = \{u_1,\ldots,u_h\}'\cap\M$ is internal: $C = [A_n]_\omega$ for some subalgebras $A_n\subset M_{d_n}$.
\end{maintheorem}

For a homomorphism of a Kazhdan group into $U(\M)$,
property~$(T)$ supplies spectral gap for a finite
generating tuple, see \cite{BHV}. Theorem~\ref{main:internality} therefore answers the original formulation of the centralizer problem affirmatively. 

For a subgroup $H<G$, define the compression semigroup
$
 P_H=\{t\in G:tHt^{-1}\le H\},
$
and call $H$ \emph{infranormal} in $G$ if $P_H$ generates $G$.
Combining Theorem~\ref{main:internality} with
\cite[Theorem~1.3]{ThomConditional} gives the following consequence.

\begin{maincorollary}
\label{main:nonhyperlinear}
Let $H<G$ be infranormal but not normal, and suppose that both
$H$ and $G$ are Kazhdan. Then the amalgamated double $G*_H G$ is
not hyperlinear.
\end{maincorollary}

The constructions in~\cite{KunThomNonsofic} supply such pairs
with both groups residually finite. An explicit example is
recorded in the final section.
The normalization result in~\cite[Theorem~1.2]{ThomConditional}
also gives a group-theoretic consequence.

\begin{maincorollary}
\label{main:centralizer-normal}
Let $H<G$ be infranormal, and suppose that both $H$ and $G$ are
Kazhdan. If $G$ is hyperlinear, then the centralizer $C_G(H)$ is
normal in $G$.
\end{maincorollary}

The proofs presented here are based
on~\cite{LiuIdeas}. This paper is organized as follows.
Section~\ref{pre:section} establishes the spectral gap
estimates, coordinate lifts of conditional expectations,
and uniform projection improvement.
Sections~\ref{pov:section} and~\ref{phy:constructionsection}
constitute the main technical part of the paper and prove
Theorem~\ref{main:correction}. 

The proof follows the strategy of Kun's work. In the graph setting,
Kun constructs an expander decomposition by repeatedly choosing
pieces with small boundary and improving their invariance until
the graph is exhausted. Section~\ref{pre:section} establishes the
corresponding improvement mechanism for projections: spectral gap
allows a projection with small boundary energy to be replaced by
a nearby projection with negligible boundary energy.

In the matrix setting, an additional difficulty arises because
the resulting projections need not be almost orthogonal.
Section~\ref{pov:section} develops an orthogonalization procedure
that retains approximate invariance and controls how much of
the space is covered.

Section~\ref{phy:constructionsection} combines these ingredients.
We successively choose projections of minimum rank with small
boundary in suitable corners and improve their invariance.
The stopping argument and the orthogonalization procedure produce
almost-invariant orthogonal blocks covering most of the space.
The minimality of the original choices then yields expansion
within these blocks, up to small errors. Finally, pruning and a
small perturbation of the doubled unitary tuple give a decomposition
into invariant blocks with a uniform spectral gap, proving
Theorem~\ref{main:correction}. We establish the pruning and
perturbation step first, in Lemma~\ref{phy:pruning}, and then
carry out the construction in the proof of
Theorem~\ref{phy:decomposition}.

In Section~\ref{cp:section}, Theorem~\ref{main:correction} first yields
an internal maximal abelian subalgebra in the relative commutant. We then reconstruct the full
relative commutant by detecting which atoms in its internal model must
be connected by matrix units. Finally, the last section records the
consequences for hyperlinear approximations of Kazhdan groups and
applies the conditional construction of~\cite{ThomConditional} to show
that the relevant group doubles from~\cite{KunThomNonsofic} are not
hyperlinear.

\subsection*{AI assistance}
AI tools contributed to several stages of the work leading to this paper.
The preceding papers~\cite{AlekseevThomCentralizers,KunThomNonsofic,ThomConditional}
already acknowledged AI assistance in drafting, editing, and checking their
manuscripts; the latter two also explicitly acknowledged its assistance in writing and simplifying the proofs.

As already explained in~\cite[Appendix~A]{LiuIdeas}, the proof in that
manuscript was obtained on 19 September 2026 using Danus~\cite{Danus}, an automated
mathematical reasoning system built on Rethlas~\cite{Rethlas} and running GPT-6 Astra.
The system carried out proof search, manuscript preparation, and automated
mathematical reviews. The recorded human instructions specified the
problem but supplied no mathematical strategy. The disclosure distinguishes
these automated reviews from formal proof certification and external
human verification.

The present exposition was developed through sustained interaction
between the first and third authors and ChatGPT/Codex, starting
from~\cite{LiuIdeas} and their earlier investigations of the centralizer problem.
AI assistance included reconstructing and simplifying proofs, exploring
alternative arguments, checking mathematical details and references,
and drafting and revising the text. The first and third author directed this revision;
the authors take responsibility for the mathematical claims, references,
and final presentation.

\section{Spectral gap and improvement of projections}
\sectionmark{Spectral gap and improvement of projections}
\label{pre:section}

\subsection{Tracial ultraproducts and spectral gap}
Throughout, $\tau$ denotes the normalized trace of the finite von Neumann algebra (often a matrix algebra)
under consideration and $\|x\|_2=\tau(x^*x)^{1/2}$, while \(\Tr\) and \(\|x\|_{\rm{HS}} = \Tr(x^*x)^{1/2} \) are reserved for the unnormalized trace and Hilbert--Schmidt norm on matrices. The same norm definitions are used for maps between finite-dimensional Hilbert spaces.

Fix a nonprincipal ultrafilter
$\omega$ on $\mathbb N$ and positive integers $d_n$. The tracial
ultraproduct is
\[
 \M=\ell^\infty(M_{d_n})/
       \{(x_n):\lim_\omega\|x_n\|_2=0\}.
\]
We use its faithful normal trace $\tau((x_n)_\omega)=
\lim_\omega\tau(x_n)$ and its trace Hilbert space $L^2(\M)$.
For unital $*$-subalgebras $A_n\subset M_{d_n}$, write
$[A_n]_\omega$ for their image in $\M$. Such an algebra is called
\emph{internal}. For a von Neumann subalgebra $A\subset\M$, $E_A$
denotes its trace-preserving conditional expectation. On $L^2$ it
is the orthogonal projection onto $L^2(A)$.

We use standard facts about finite von Neumann algebras:
the spectral and polar decomposition theorems, trace-preserving
conditional expectations, and the finite tracial ultraproduct construction.
Finite-dimensional complete positivity, the Kraus representation,
and Schwarz's inequality are also used. We occasionally call an orthogonal family of projections summing to $1$ a projection partition.

Fix $h\ge1$ and unitary matrices $u_{1,n},\ldots,u_{h,n} \in U(d_n)$.
Put $u_j=(u_{j,n})_\omega$ and
\begin{equation}\label{pre:lazy-markov}
 C=\{u_1,\ldots,u_h\}'\cap\M,
 \qquad
 T_n(x)=\frac12x+\frac1{4h}\sum_{j=1}^h
              (u_{j,n}xu_{j,n}^*+u_{j,n}^*xu_{j,n}).
\end{equation}
The induced operator on $L^2(\M)$ is denoted by $T$. We will call these operators lazy Markov operators associated with the corresponding unitary tuples. Each $T_n$
is unital, completely positive (u.c.p.) and trace-preserving, self-adjoint on $L^2$, and satisfies $0\le T_n\le I$
as a Hilbert-space operator. Define the \emph{boundary energy} of $x$ by
\begin{equation}\label{pre:energy}
 \mathcal E_n(x)=\frac1{4h}\sum_j\|[u_{j,n},x]\|_2^2
                  =\langle(I-T_n)x,x\rangle.
\end{equation}
The identity follows by expanding the squared commutators. 

Our main assumption is \emph{spectral gap in the ultraproduct}: for the same $\kappa\in(0,1)$,
\begin{equation}\label{pre:gap}
 \langle(I-T)\xi,\xi\rangle\ge\kappa\|\xi\|_2^2
 \qquad(\xi\perp L^2(C)).
\end{equation}

The fixed space of $T$ is exactly $L^2(C)$. We obtain
\begin{equation}\label{pre:spectralbounds}
 \|T^j-E_C\|_{2\to2}\le (1-\kappa)^j,
 \qquad \|T^k-T^j\|_{2\to2}\le (1-\kappa)^j\quad(1\le j\le k).
\end{equation}

\Needspace{8\baselineskip}
\subsection{Uniform transfer and coordinate lifts}

\begin{lemma}\label{pre:transfer}
Suppose linear maps $F_n:M_{d_n}\to M_{d_n}$ are uniformly bounded
in both $\infty\to\infty$ and $2\to2$ norms, and induce $F_\omega$
on $\M$. Then
\[
 \|F_\omega\|_{\infty\to2}
       =\lim_\omega\|F_n\|_{\infty\to2}.
\]
\end{lemma}
\begin{proof}
Contraction representatives give the inequality from left to right.
For the reverse inequality choose a maximizing contraction in each
finite-dimensional coordinate and take its class. The $2\to2$ bound
ensures that the induced map is well-defined.
\end{proof}

\begin{proposition}
\label{pre:expectation}
Suppose the tuple $(u_1,\dots,u_h)$ has spectral gap in the ultraproduct. Then there are integers $\ell_n\to_\omega\infty$
such that $\Phi_n:=T_n^{\ell_n}$ induces $E_C$ on every bounded
sequence. The same holds for $T_n^{k_n}$ whenever $1\le k_n\le\ell_n$
and $k_n\to_\omega\infty$. Moreover, for any unital $*$-subalgebras
$A_n\subset M_{d_n}$,
\[
 [A_n]_\omega=C\quad\Longleftrightarrow\quad
 \|\Phi_n-E_{A_n}\|_{\infty\to2}\longrightarrow_\omega0.
\]
\end{proposition}
\begin{proof}
By Lemma~\ref{pre:transfer} and \eqref{pre:spectralbounds}, for each
fixed $1\le j\le k$,
\[
 \lim_\omega\|T_n^k-T_n^j\|_{\infty\to2}\le (1-\kappa)^j.
\]
Choose decreasing sets $W_m\in\omega$, contained in
$\{n:n\ge m\}$, on which simultaneously
\[
 \|T_n^k-T_n^j\|_{\infty\to2}\le (1-\kappa)^j+1/m
          \quad(1\le j\le k\le m).
\]
Only finitely many conditions are imposed at each stage. Set
$\ell_n=\max(\{m\le n:n\in W_m\}\cup\{1\})$.
Then $\ell_n\to_\omega\infty$, and for each fixed $j$,
\[
 \lim_\omega\|T_n^{\ell_n}-T_n^j\|_{\infty\to2}\le (1-\kappa)^j.
\]
The induced map $\Phi$ therefore satisfies
$\|\Phi-E_C\|_{\infty\to2}\le2(1-\kappa)^j$ for every $j$, so equals $E_C$.
The same estimates apply to every $k_n\le\ell_n$ tending to infinity.

Coordinate expectations induce the expectation onto $[A_n]_\omega$:
their outputs have bounded representatives in $A_n$, they fix all
such representatives, and their trace orthogonality passes to the
quotient. Apply Lemma~\ref{pre:transfer} to $\Phi_n-E_{A_n}$ to
obtain both implications of the final assertion.
\end{proof}

\subsection{Matrix norm and spectral estimates}
We will now recall some basic arguments that we will use repeatedly:

We use the \emph{Hilbert--Schmidt Lipschitz estimate}
\begin{equation}\label{pre:hs-lipschitz}
 \|g(A)-g(B)\|_2\le\operatorname{Lip}(g)\|A-B\|_2
\end{equation}
for self-adjoint matrices $A,B$ and scalar Lipschitz functions $g$.
It follows by expanding in eigenbases of $A$ and $B$.
From this, one can obtain the \emph{Araki--Yamagami inequality}
\cite[Theorem~1]{ArakiYamagami} for rectangular matrices,
\begin{equation}\label{pre:absolute-value}
 \||X|-|Y|\|_2^2\le2\|X-Y\|_2^2.
\end{equation}

We will apply the Hilbert--Schmidt Lipschitz estimate also to spectral cuts: Indeed, clipping singular values at one is the
Hilbert--Schmidt metric projection onto the operator-norm unit ball;
it is $1$-Lipschitz and commutes with left and right unitary multiplication.

\begin{lemma}[Spectral coarea]\label{pre:coarea}
For a unitary tuple $u$, put
$\mathcal E_u(x)=(4h)^{-1}\sum_j\|[u_j,x]\|_2^2$.
For self-adjoint $x$ and $a<b$,
\begin{equation}\label{pre:global-coarea}
 \int_a^b\mathcal E_u(1_{[s,\infty)}(x))\,ds
 \le\tfrac12\sqrt{\mathcal E_u(x)}.
\end{equation}
If $(x_i)$ are positive contractions of total trace
$t=\sum_i\tau(x_i)$, then, for $0<a<b\le1$,
\[
 \int_a^b\sum_i\mathcal E_u(1_{[s,1]}(x_i))\,ds
 \le\left(\frac{t}{2a}\sum_i\mathcal E_u(x_i)\right)^{1/2}.
\]
Consequently, in the first case there exists $s\in[a,b]$ such that
\[
 \mathcal E_u(1_{[s,\infty)}(x))
 \le\frac{\sqrt{\mathcal E_u(x)}}{2(b-a)},
\]
and in the second case there exists $s\in[a,b]$ such that
\[
 \sum_i\mathcal E_u(1_{[s,1]}(x_i))
 \le\frac1{b-a}
       \left(\frac{t}{2a}\sum_i\mathcal E_u(x_i)\right)^{1/2}.
\]

\end{lemma}

\begin{proof}
Choose an orthonormal eigenbasis of $x$, with eigenvalues
$\lambda_1,\ldots,\lambda_d$, and put
\[
 c_{k\ell}=\frac1{4hd}\sum_{j=1}^h|(u_j)_{k\ell}|^2.
\]
Since each $u_j$ is unitary, these nonnegative weights satisfy
\[
 \sum_\ell c_{k\ell}=\frac1{4d},
 \qquad
 \sum_k c_{k\ell}=\frac1{4d},
 \qquad
 \sum_{k,\ell}c_{k\ell}=\frac14.
\]
Computing the entries of the commutators gives
$
 \mathcal E_u(x)
 =\sum_{k,\ell}c_{k\ell}|\lambda_k-\lambda_\ell|^2.
$
Likewise, for $p_s=1_{[s,\infty)}(x)$,
\[
 \mathcal E_u(p_s)
 =\sum_{k,\ell}c_{k\ell}
   |1_{\{\lambda_k\ge s\}}-1_{\{\lambda_\ell\ge s\}}|^2.
\]
The two indicators differ precisely when $s$ lies between
$\lambda_k$ and $\lambda_\ell$, up to endpoints. Consequently,
\[
 \int_a^b
 |1_{\{\lambda_k\ge s\}}-1_{\{\lambda_\ell\ge s\}}|^2\,ds
 \le|\lambda_k-\lambda_\ell|.
\]
Summing and applying Cauchy--Schwarz, we obtain
\[
 \begin{aligned}
 \int_a^b\mathcal E_u(p_s)\,ds
 &\le\sum_{k,\ell}c_{k\ell}|\lambda_k-\lambda_\ell|
 \le
 \left(\sum_{k,\ell}c_{k\ell}\right)^{1/2}
 \left(\sum_{k,\ell}c_{k\ell}
                  |\lambda_k-\lambda_\ell|^2\right)^{1/2}
 =\tfrac12\sqrt{\mathcal E_u(x)}.
 \end{aligned}
\]
This proves \eqref{pre:global-coarea}.

For the family estimate, choose an eigenbasis separately for
each $x_i$, and write $\lambda_{i,k}\in[0,1]$ and
$c^{(i)}_{k\ell}$ for the corresponding eigenvalues and weights.
Put
\[
 I_i=\{k:\lambda_{i,k}\ge a\},
 \qquad
 S_i=\{(k,\ell):k\in I_i\text{ or }\ell\in I_i\}.
\]
If $(k,\ell)\notin S_i$, both spectral indicators vanish
throughout $[a,b]$. Moreover, the row and column sum identities
give
\[
 \begin{aligned}
 \sum_{(k,\ell)\in S_i}c^{(i)}_{k\ell}
 &\le
 \sum_{k\in I_i}\sum_\ell c^{(i)}_{k\ell}
 +\sum_{\ell\in I_i}\sum_k c^{(i)}_{k\ell}
 =\frac{|I_i|}{2d}
 =\tfrac12\tau(1_{[a,1]}(x_i))
 \le\frac{\tau(x_i)}{2a},
 \end{aligned}
\]
where the last inequality follows from
$a1_{[a,1]}(x_i)\le x_i$.
Applying the preceding indicator estimate and then
Cauchy--Schwarz jointly over all $i,k,\ell$, we obtain
\[
 \begin{aligned}
 \int_a^b\sum_i\mathcal E_u(1_{[s,1]}(x_i))\,ds
 &\le
 \sum_i\sum_{(k,\ell)\in S_i}
 c^{(i)}_{k\ell}|\lambda_{i,k}-\lambda_{i,\ell}|\\
 &\le
 \left(\sum_i\sum_{(k,\ell)\in S_i}
                   c^{(i)}_{k\ell}\right)^{1/2}
 \left(\sum_i\sum_{(k,\ell)\in S_i}
 c^{(i)}_{k\ell}|\lambda_{i,k}-\lambda_{i,\ell}|^2
 \right)^{1/2}\\
 &\le
 \left(\frac{t}{2a}\sum_i\mathcal E_u(x_i)\right)^{1/2}.
 \end{aligned}
\]
Finally, averaging over the interval $[a,b]$ proves both
existence assertions. For the family, we average the sum of
the energies, obtaining a single threshold $s$ for all $i$.
\end{proof}

\subsection{Making projections almost invariant}
\label{pre:repairsection}

The following theorem is the analogue of Kun's projection-improvement
argument~\cite[Corollary~9 and Proposition~11]{KunExpanders}. Notice
that this variant does not use the full power of property~$(T)$ as in
Kun's original approach, but only spectral gap in the ultraproduct,
similarly to~\cite[Proposition~5.10]{AlekseevDrigalla}.

\begin{theorem}\label{pre:repair}
Suppose the tuple $(u_1,\dots,u_h)$ has spectral gap in the ultraproduct. Then there are projections $r_n$ and positive
numbers $\alpha_n\to_\omega0$ such that $\tau(1-r_n)\le\alpha_n$,
and every projection $p\le r_n$ with
$\mathcal E_n(p)\le\kappa\tau(p)/64$ admits a projection
$f\in M_{d_n}$ satisfying
\begin{equation}\label{pre:repairbounds}
 \|f-p\|_2^2\le36\kappa^{-1}\mathcal E_n(p),\qquad
 \tfrac13\tau(p)\le\tau(f)\le\tfrac53\tau(p),\qquad
 \mathcal E_n(f)\le\alpha_n\tau(p).
\end{equation}
\end{theorem}

\begin{proof}
We first claim that there exist positive $\beta_n\to_\omega0$ such that for every contraction $x\in M_{d_n}$,
\[
 \|x-\Phi_n(x)\|_2
 \le\kappa^{-1/2}\mathcal E_n(x)^{1/2}+\beta_n,
 \qquad
 \mathcal E_n(\Phi_n(x))\le\beta_n^2.
\]
Indeed, assume for the sake of contradiction that there exists $\beta > 0$ and some sequence of contractions $(x_n)$ satisfying
\begin{equation}\label{eq:bad-xn}
 \|x_n-\Phi_n(x_n)\|_2
\ge\kappa^{-1/2}\mathcal E_n(x_n)^{1/2}+\beta
\quad\text{or}\quad
\mathcal E_n(\Phi_n(x_n))\ge\beta^2.
\end{equation}
along the ultrafilter. Using the maps $\Phi_n=T_n^{\ell_n}$ from
Proposition~\ref{pre:expectation} which induce $E_C$ on every
bounded sequence and writing
\[
 \mathcal E_\omega(x)
 =\frac1{4h}\sum_{j=1}^h\|[u_j,x]\|_2^2
 \qquad(x\in\M),
\]
we see that with $x=(x_n)_\omega$, we have 
\(
 \lim_\omega\mathcal E_n(x_n)=\mathcal E_\omega(x).
\)
Since $(\Phi_n(x_n))_\omega=E_C(x)$, we also have
\[
 \lim_\omega\|x_n-\Phi_n(x_n)\|_2
 =\|x-E_C(x)\|_2\le\kappa^{-1/2}\mathcal E_\omega(x)^{1/2},
 \qquad
 \lim_\omega\mathcal E_n(\Phi_n(x_n))
 =\mathcal E_\omega(E_C(x))=0.
\]
This contradicts the defining property \ref{eq:bad-xn} of $(x_n)$.

Now set
\begin{equation}\label{pre:repair-sequence}
 \alpha_n=8\sqrt{\beta_n}+\ell_n^{-1}+n^{-1}.
\end{equation}
Then $\alpha_n>0$ and $\alpha_n\to_\omega0$. Moreover,
$\alpha_n$ dominates $\beta_n$, $\ell_n^{-1}$ and $n^{-1}$,
and
\(
 \beta_n\le {\alpha_n^2}/{64}.
\)

Fix a coordinate and temporarily suppress the subscript $n$.
Call a nonzero projection $p$ \emph{distance-bad} if
\[
 \mathcal E(p)\ge\alpha\tau(p),
 \qquad
 \|p-\Phi(p)\|_2
 >2\kappa^{-1/2}\mathcal E(p)^{1/2}.
\]
We claim that every orthogonal family $(p_i)_{i=1}^m$ of
distance-bad projections satisfies
\[
 \sum_i\tau(p_i)\le\frac{\kappa\beta^2}{\alpha}.
\]
To prove this, let $\epsilon_1,\ldots,\epsilon_m$ be independent
uniform signs and put $X_\epsilon=\sum_i\epsilon_i p_i$.
Orthogonality gives
$X_\epsilon^*X_\epsilon=\sum_i p_i\le1$, so $X_\epsilon$
is a contraction. Since
$\mathbb E_\epsilon(\epsilon_i\epsilon_j)=0$ for $i\ne j$,
linearity of $\Phi$ and of the commutators gives
\[
 \mathbb E_\epsilon
 \|X_\epsilon-\Phi(X_\epsilon)\|_2^2
 =\sum_i\|p_i-\Phi(p_i)\|_2^2,
 \qquad
 \mathbb E_\epsilon\mathcal E(X_\epsilon)
 =\sum_i\mathcal E(p_i).
\]
Apply the defining estimate for $\beta$ to $X_\epsilon$
and use Minkowski's inequality on the finite probability
space of signs. We obtain
\[
 \begin{aligned}
 \left(\sum_i\|p_i-\Phi(p_i)\|_2^2\right)^{1/2}
 &\le
 \left(\mathbb E_\epsilon
 \bigl(\kappa^{-1/2}\mathcal E(X_\epsilon)^{1/2}
       +\beta\bigr)^2\right)^{1/2}
 \le
 \kappa^{-1/2}
 \left(\sum_i\mathcal E(p_i)\right)^{1/2}+\beta.
 \end{aligned}
\]
Furthermore, distance-badness gives
\[
 \left(\sum_i\|p_i-\Phi(p_i)\|_2^2\right)^{1/2}
 >2\kappa^{-1/2}
 \left(\sum_i\mathcal E(p_i)\right)^{1/2}.
\]
Combining these inequalities yields
$\sum_i\mathcal E(p_i)\le\kappa\beta^2$.
Since $\alpha\tau(p_i)\le\mathcal E(p_i)$, the claimed
trace bound follows. 

Independently, call a nonzero projection $q$
\emph{energy-bad} if
\[
 \mathcal E(\Phi(q))>\frac{\alpha^2}{36}\tau(q).
\]
For an orthogonal family $(q_i)$ of energy-bad projections,
the random-sign sum $Y_\epsilon=\sum_i\epsilon_i q_i$
is again a contraction. Cancellation of the mixed terms
and the second defining estimate for $\beta$ give
\[
 \sum_i\mathcal E(\Phi(q_i))
 =\mathbb E_\epsilon\mathcal E(\Phi(Y_\epsilon))
 \le\beta^2.
\]
It follows that
\(
 \sum_i\tau(q_i)\le{36\beta^2}/{\alpha^2}.
\)

Now for each $n$ choose a maximal orthogonal family of distance-bad projections and, separately,
a maximal orthogonal family of energy-bad projections. Let $b_n$ and $c_n$ be their respective sums, and put
\(
 r_n=1-(b_n\vee c_n).
\)
Since the trace of the join is at most the sum of
the traces, the preceding estimates imply
\[
 \tau(1-r_n)
 \le\tau(b_n)+\tau(c_n)
 \le\frac{\kappa\beta_n^2}{\alpha_n}
       +\frac{36\beta_n^2}{\alpha_n^2} < \alpha_n.
\]
By maximality, no nonzero projection under $r_n$ is bad
in either sense.  

Now let $p\le r_n$ satisfy
$\mathcal E_n(p)\le\kappa\tau(p)/64$. If $p=0$, take $f=0$. If $\mathcal E_n(p)<\alpha_n\tau(p)$, take $f=p$;
all the required estimates then hold immediately, so assume now that
$\mathcal E_n(p)\ge\alpha_n\tau(p)$.
Put $z=\Phi_n(p)$. Since $\Phi_n$ is unital, positive,
and trace-preserving, $z$ is a positive contraction with
$\tau(z)=\tau(p)$. The fact that $p$ is neither
distance-bad nor energy-bad gives
\begin{equation}\label{pre:closeaverage}
 \|z-p\|_2^2\le4\kappa^{-1}\mathcal E_n(p),
 \qquad
 \mathcal E_n(z)\le\alpha_n^2\tau(p)/36.
\end{equation}
Apply the positive-contraction estimate of
Lemma~\ref{pre:coarea} to $z$ on $[1/3,2/3]$.
There is an $s$ in this interval such that
$f_s=1_{[s,1]}(z)$ satisfies
\[
 \mathcal E_n(f_s)
 \le3\left(\frac32\tau(z)\mathcal E_n(z)\right)^{1/2}
 \le6\sqrt{\tau(z)\mathcal E_n(z)}
 \le\alpha_n\tau(p).
\]
It remains to control the distance from $f_s$ to $p$.
We show that every $s\in[1/3,2/3]$ satisfies
\[
 \|f_s-p\|_2^2\le9\|z-p\|_2^2.
\]
To this end, choose an orthonormal eigenbasis $(\xi_i)_{i=1}^{d_n}$ of
$z$, with eigenvalues $\lambda_i\in[0,1]$, and write
$b_i=\langle p\xi_i,\xi_i\rangle\in[0,1]$.
Using $p^2=p$ and expanding the traces, we obtain
\[
 d_n\|z-p\|_2^2
 =\sum_i\bigl(\lambda_i^2(1-b_i)
                  +(1-\lambda_i)^2b_i\bigr),
\]
whereas
\[
 d_n\|f_s-p\|_2^2
 =\sum_{\lambda_i\ge s}(1-b_i)
   +\sum_{\lambda_i<s}b_i.
\]
If $\lambda_i\ge s$, then $\lambda_i\ge1/3$, so
\[
 1-b_i
 \le9\lambda_i^2(1-b_i)
 \le9\bigl(\lambda_i^2(1-b_i)
                  +(1-\lambda_i)^2b_i\bigr).
\]
If $\lambda_i<s$, then $\lambda_i<2/3$, and similarly
\[
 b_i
 \le9(1-\lambda_i)^2b_i
 \le9\bigl(\lambda_i^2(1-b_i)
                  +(1-\lambda_i)^2b_i\bigr).
\]
Summing proves the claimed distance estimate.
For the threshold chosen above, set $f=f_s$.
Together with \eqref{pre:closeaverage}, this yields
\[
 \|f-p\|_2^2
 \le9\|z-p\|_2^2
 \le36\kappa^{-1}\mathcal E_n(p).
\]
Finally, for any two projections $f,p$,
\[
 \|f-p\|_2^2
 =\tau(f)+\tau(p)-2\tau(fp)
 \ge|\tau(f)-\tau(p)|,
\]
because $\tau(fp)\le\min\{\tau(f),\tau(p)\}$. Therefore
\[
 |\tau(f)-\tau(p)|
 \le36\kappa^{-1}\mathcal E_n(p)
 \le\frac9{16}\tau(p),
\]
and hence
\[
 \frac7{16}\tau(p)\le\tau(f)\le\frac{25}{16}\tau(p),
\]
which implies the claimed bounds. This finishes the proof.
\end{proof}

\section{Orthogonalization of projections with small boundary energy}
\sectionmark{Orthogonalization of projections with small boundary energy}
\label{pov:section}

In this section we collect some preparatory statements which will be used in the proof of the main decomposition theorem. The main goal is to ensure that we can ``orthogonalize'' a maximal almost-invariant projection family keeping the properties of it being almost-invariant and exhaustive (as summarized in Proposition \ref{phy:summary}), and this will be an important step in the course of the proof of Theorem \ref{phy:decomposition}.

The orthogonalization method used below should be compared to
de la Salle's orthogonalization of positive operator valued measures
\cite[Theorem~2 and Section~3]{DeLaSalle}.

All $O(\cdot)$ constants in this section are absolute, also independent
of $h$. Bounds involving $\log(1/s)$ are understood for sufficiently
small $s>0$ and as zero at $s=0$.

\begin{definition}\label{pov:defect}
For a finite family $\boldsymbol q=(q_i)_i$ of projections in $M_d$, put
\begin{equation}\label{pov:covering-defect}
 \delta(\boldsymbol q)=
 \tau\!\left(\left(1-\Bigl(\sum_i q_i\Bigr)^{1/2}\right)_+^2\right).
\end{equation}
We call this the \emph{coverage defect} of the family. It is a one-sided
defect: it vanishes exactly when $\sum_i q_i\ge1$. For an orthogonal
family, $\delta(\boldsymbol q)=\tau(1-\sum_i q_i)$.
\end{definition}

It is possible for a family $\boldsymbol q$ of projections of total trace $1$ to satisfy $\bigvee_i q_i = 1$ and still have large coverage defect: this would happen if the projections are far from being orthogonal. In fact, this is exactly the problem which requires an additional step in the construction compared to the expander decomposition in the sofic case. This is where the next step enters.

\subsection{Orthogonalization of projections using resolvents}

We first note a useful bound on the boundary energy of the resolvent.

\begin{lemma}\label{phy:resolvents}
Let $f_1,\ldots,f_m\in M_d$ be projections and put
$S_{i-1}=\sum_{k<i}f_k$. For every unitary $u$ and $0<\lambda\le1$,
\begin{equation}\label{phy:resolvent-bound}
 \sum_i\|[u,(\lambda+S_{i-1})^{-1}]f_i\|_2^2
 \le C\lambda^{-3}\sum_i\|[u,f_i]\|_2^2,
\end{equation}
where $C$ is an absolute constant.
\end{lemma}

\begin{proof}
We first assume that $u=u^*$, so that $u^2=1$.
For a positive invertible matrix $A$, put
$
 H(A):=\tau(AuA^{-1}u)-1.
$
Expanding the square and using cyclicity of the trace gives
\[
 H(A)
 =\frac12\|A^{-1/2}[u,A]A^{-1/2}\|_2^2\ge0.
\]
In particular, $H(\lambda1)=0$.

Fix $i$, write $f=f_i$, and set
\[
 A_t=\lambda+S_{i-1}+tf,\qquad
 D_t=A_t^{-1},\qquad B_t=uA_tu,
 \qquad K_t=D_t-B_t^{-1}
 \quad(0\le t\le1).
\]
Also put
$\varepsilon_i=\|[u,f]\|_2$, and
$z_i(t)=\|[u,D_t]f\|_2.$
We have $A_t,B_t\ge\lambda1$ and
$\|D_t\|\le\lambda^{-1}$.

Using $D_t'=-D_tfD_t$ and cyclicity of the trace, we obtain
\[
 \frac{d}{dt}H(A_t)
 =\tau(fB_t^{-1})-\tau(fD_tB_tD_t).
\]
Since $D_t=B_t^{-1}+K_t$, we get
$
 D_tB_tD_t=B_t^{-1}+2K_t+K_tB_tK_t.
$
Furthermore, direct expansion gives
\[
 \tau([u,f]^*[u,D_t])=2\tau(fK_t).
\]
Consequently,
\[
 \frac{d}{dt}H(A_t)
 =-\operatorname{Re}\tau([u,f]^*[u,D_t])
   -\|B_t^{1/2}K_tf\|_2^2.
\]
As $K_t=u[u,D_t]$ and $B_t\ge\lambda1$, the second term
satisfies
\[
 \|B_t^{1/2}K_tf\|_2^2\ge\lambda z_i(t)^2.
\]

To estimate the first term, observe that $[u,f]$ is off
diagonal relative to $f$. Only the off-diagonal part of
$[u,D_t]$ therefore contributes to the inner product.
Since $[u,D_t]$ is skew-adjoint, its two off-diagonal blocks
have equal $2$-norm. Hence
\[
 \begin{aligned}
 |\tau([u,f]^*[u,D_t])|
 &\le\sqrt2\,\varepsilon_i\|(1-f)[u,D_t]f\|_2\le\sqrt2\,\varepsilon_i z_i(t)\le\lambda^{-1}\varepsilon_i^2
       +\frac{\lambda}{2}z_i(t)^2.
 \end{aligned}
\]
It follows that
\[
 \frac{d}{dt}H(A_t)
 \le\lambda^{-1}\varepsilon_i^2
       -\frac{\lambda}{2}z_i(t)^2.
\]
The endpoint of step $i$ is the starting point of step $i+1$.
Integrating and summing therefore gives
\[
 H\left(\lambda1+\sum_i f_i\right)
 \le\lambda^{-1}\sum_i\varepsilon_i^2
       -\frac{\lambda}{2}
        \sum_i\int_0^1z_i(t)^2\,dt.
\]
Since $H$ is nonnegative,
\[
 \sum_i\int_0^1z_i(t)^2\,dt
 \le2\lambda^{-2}\sum_i\varepsilon_i^2.
\]

We next compare the integral with the values at $t=0$.
The resolvent identity gives
\[
 D_0=D_t+tD_0fD_t.
\]
Taking commutators and multiplying on the right by $f$,
we obtain
\[
 \begin{aligned}
 [u,D_0]f
 ={}&[u,D_t]f+t[u,D_0]fD_tf+tD_0[u,f]D_tf+tD_0f[u,D_t]f.
 \end{aligned}
\]
Using $\|D_0\|,\|D_t\|\le\lambda^{-1}$, this implies
\[
 z_i(0)
 \le\left(1+\frac{t}{\lambda}\right)z_i(t)
       +\frac{t}{\lambda}z_i(0)
       +\frac{t}{\lambda^2}\varepsilon_i.
\]
For $0\le t\le\lambda/2$, we can absorb the term involving
$z_i(0)$ on the right and obtain
\[
 z_i(0)\le3z_i(t)+\lambda^{-1}\varepsilon_i.
\]
Squaring and averaging over $[0,\lambda/2]$ yields
\[
 z_i(0)^2
 \le\frac{36}{\lambda}\int_0^{\lambda/2}z_i(t)^2\,dt
       +2\lambda^{-2}\varepsilon_i^2.
\]
Summing over $i$ and using $\lambda\le1$, we conclude that
\[
 \sum_i z_i(0)^2
 \le(72\lambda^{-3}+2\lambda^{-2})
       \sum_i\varepsilon_i^2
 \le74\lambda^{-3}\sum_i\varepsilon_i^2.
\]

For a general unitary $u$, apply the preceding argument in
$M_{2d}$ to
\[
 \widetilde u=
 \begin{pmatrix}0&u\\u^*&0\end{pmatrix},
 \qquad
 \widetilde f_i=f_i\oplus f_i,
\]
using the normalized trace on $M_{2d}$.
The associated resolvents are
$\widetilde D_i=D_i\oplus D_i$, where
$D_i=(\lambda+S_{i-1})^{-1}$.
Thus
\[
 \|[\widetilde u,\widetilde D_i]\widetilde f_i\|_{2,2d}^2
 =\frac12\left(
 \|[u,D_i]f_i\|_{2,d}^2+
 \|[u^*,D_i]f_i\|_{2,d}^2\right)
 \ge\frac12\|[u,D_i]f_i\|_{2,d}^2.
\]
On the other hand,
\[
 \|[\widetilde u,\widetilde f_i]\|_{2,2d}^2
 =\|[u,f_i]\|_{2,d}^2,
\]
because $[u^*,f_i]=-u^*[u,f_i]u^*$.
This proves \eqref{phy:resolvent-bound}, for example with
$C=148$.
\end{proof}

The next lemma improves coverage of a family of projections, effectively (almost) orthogonalizing them similarly to the Gram--Schmidt process. This happens through using resolvent differences, the main observation being that if $S_i$ are increasing sums of projections, then for small $\lambda > 0$ the resolvent difference
\[
\lambda((\lambda +S_{i-1})^{-1}-(\lambda +S_{i})^{-1})
\]
strongly converges to $\supp S_i - \supp S_{i-1}$. Therefore taking appropriate spectral projections of these resolvents yields a family with small coverage defect.

\begin{lemma}\label{phy:covariance}
Let $u=(u_1,\ldots,u_h)\in U(d)^h$, let
$e,f_1,\ldots,f_m\in M_d$ be projections, and let $0<\gamma\le1/4$.
Suppose
\[
 \tau(1-e)+\sum_i\tau(f_i)\le3,\qquad
 \mathcal E_u(1-e)+\sum_i\mathcal E_u(f_i)\le\gamma^{16}.
\]
Put $S_0=1-e$ and $S_i=S_0+\sum_{j\le i}f_j$.
There exist an absolute constant $C$ and a family
$\boldsymbol q=(q_i)_i$ of projections such that
\begin{equation}\label{phy:resolvent-projections}
 \rank q_i\le\rank f_i,\qquad
 \sum_i\mathcal E_u(q_i)\le C\gamma^2,\qquad
 \sum_i\tau(S_{i-1}^2q_i)\le3\gamma,
\end{equation}
and
\begin{equation}\label{phy:resolvent-covering}
 \delta(\boldsymbol q)
 \le\tau(1-e)+\tau(1_{[0,\gamma]}(S_m))+7\gamma.
\end{equation}
\end{lemma}
\begin{proof}
Put $\lambda=\gamma^2$ and define
\[
 R_i=(\lambda+S_i)^{-1},\qquad H_i=1+f_iR_{i-1}f_i,
 \qquad X_i=\sqrt\lambda R_{i-1}f_iH_i^{-1/2}.
\]
The resolvent identity gives positive contractions
\begin{equation}\label{phy:resolvent-increments}
 D_i=\lambda(R_{i-1}-R_i)=X_iX_i^*.
\end{equation}
Thus $\rank D_i\le\rank f_i$ and
\begin{equation}\label{phy:resolvent-telescope}
 \sum_i D_i=\lambda(R_0-R_m)\le1.
\end{equation}
Also $D_i\le\lambda R_{i-1}f_iR_{i-1}$ and
$0\le S_{i-1}R_{i-1}\le1$, so
\begin{equation}\label{phy:resolvent-moment}
 \sum_i\tau(S_{i-1}^2D_i)
 \le\lambda\sum_i\tau(f_i)\le3\lambda.
\end{equation}

Fix a unitary $w$, and abbreviate
\[
 \varepsilon_i=\|[w,f_i]\|_2,\qquad
 z_i=\|[w,R_{i-1}]f_i\|_2.
\]
Expanding the commutator gives
\[
 [w,H_i]
 =[w,f_i]R_{i-1}f_i
   +f_i[w,R_{i-1}]f_i
   +f_iR_{i-1}[w,f_i],
\]
and hence
$\|[w,H_i]\|_2\le2\lambda^{-1}\varepsilon_i+z_i$.
Since $H_i\ge1$, the function $x\mapsto x^{-1/2}$ is
$1/2$-Lipschitz on its spectrum.
Applying the Hilbert--Schmidt Lipschitz estimate to
$wH_iw^*$ and $H_i$ therefore gives
\[
 \|[w,H_i^{-1/2}]\|_2
 \le\tfrac12\|[w,H_i]\|_2
 \le\lambda^{-1}\varepsilon_i+\tfrac12z_i.
\]
The product rule, together with
$\|R_{i-1}\|\le\lambda^{-1}$,
$\|H_i^{-1/2}\|\le1$ and $\lambda\le1$, yields
\[
 \begin{aligned}
 \|[w,X_i]\|_2
 &\le\sqrt\lambda\,z_i+\lambda^{-1/2}\varepsilon_i
       +\lambda^{-1/2}\|[w,H_i^{-1/2}]\|_2\le C\bigl(\lambda^{-3/2}\varepsilon_i
                   +\lambda^{-1/2}z_i\bigr).
 \end{aligned}
\]
Prepending $1-e$ to the input family in
Lemma~\ref{phy:resolvents} gives
\[
 \sum_i z_i^2
 \le C\lambda^{-3}
       \left(\|[w,1-e]\|_2^2+\sum_i\varepsilon_i^2\right).
\]
Consequently,
\[
 \begin{aligned}
 \sum_i\|[w,X_i]\|_2^2
 &\le C\lambda^{-3}\sum_i\varepsilon_i^2
       +C\lambda^{-1}\sum_i z_i^2\le C\lambda^{-4}
       \left(\|[w,1-e]\|_2^2+\sum_i\varepsilon_i^2\right).
 \end{aligned}
\]
Finally, $D_i=X_iX_i^*$ and $\|X_i\|\le1$ imply
\[
 \|[w,D_i]\|_2
 \le\|[w,X_i]\|_2+\|[w,X_i^*]\|_2
 =2\|[w,X_i]\|_2.
\]
Applying these estimates to $w=u_j$ and summing over $j$,
we obtain
\[
\sum_i\mathcal E_u(D_i)
 \le4\sum_i\mathcal E_u(X_i)
 \le C\lambda^{-4}
       \left(\mathcal E_u(1-e)+\sum_i\mathcal E_u(f_i)\right)
 \le C\lambda^{-4}\gamma^{16}.
\]

Their total trace is at most one by \eqref{phy:resolvent-telescope}.
Lemma~\ref{pre:coarea} on $[\gamma,2\gamma]$ therefore supplies a
common $t$; define $\boldsymbol q=(q_i)_i$ by
$q_i=1_{[t,1]}(D_i)$. Then
\[
 \sum_i\mathcal E_u(q_i)
 \le C\gamma^{-3/2}\lambda^{-2}\gamma^8
 \le C\gamma^2.
\]
The rank assertion is immediate, and $q_i\le\gamma^{-1}D_i$
gives the last bound in \eqref{phy:resolvent-projections}.
Finally, $\tau((D_i-q_i)_+)\le2\gamma\tau(f_i)$.
Using $(1-\sqrt t)_+^2\le(1-t)_+$ and trace subadditivity of
the positive part,
\[
 \delta(\boldsymbol q)
 \le\tau\Bigl(1-\sum_i D_i\Bigr)
                 +\sum_i\tau((D_i-q_i)_+)
 \le\tau(1-e)+\tau(\lambda R_m)+6\gamma.
\]
Since $\lambda=\gamma^2$, the scalar inequality
$\lambda/(\lambda+x)\le1_{[0,\gamma]}(x)+\gamma$ for $x\ge0$
proves \eqref{phy:resolvent-covering}.
\end{proof}

\subsection{Correcting a covering family to an orthogonal one}

Let $H$ and $K$ be finite-dimensional Hilbert spaces, put
$d=\dim H$, and let $u=(u_1,\ldots,u_h)\in U(H)^h$ and
$v=(v_1,\ldots,v_h)\in U(K)^h$. For $Y\in B(H,K)$, define the intertwining defect
\[
 \mathcal I_{u,v}(Y)=\frac1{4h}\sum_j\|v_jY-Yu_j\|_2^2.
\]

The next lemma estimates the cost of replacing a partial isometry with one whose range projection exactly commutes with a given partition projection in terms of the relevant intertwining defects.

\begin{lemma}\label{pov:selection}
Let $H=\mathbb C^d$, $\dim K\le3d$, and let $u_j\in U(H)$,
$v_j\in U(K)$, $1\le j\le h$. Let $(E_i)_i$ be a projection
partition of $K$ commuting with every $v_j$.
If $Y:H\to K$ is a partial isometry, then, for every $t\ge1$,
there is a partial isometry $Z:H\to K$ such that
\begin{equation}\label{pov:finite-selection}
 Z^*Z=Y^*Y,\qquad [ZZ^*,E_i]=0,\qquad
 \mathcal I_{u,v}(Z)\le
 2e^{4t}\mathcal I_{u,v}(Y)+32t^{-1/2}.
\end{equation}
\end{lemma}
\begin{proof}
Choose independent $s_i$ uniformly on $[-1,1]$, put
$A=\sum_i s_iE_i$, and let $Y_A$ be the polar part of $e^{tA}Y$.
Its initial projection is $Y^*Y$; write $p_A=Y_AY_A^*$. Set
$B=\bigoplus_i E_iB(K)E_i$, and let $E_B:B(K)\to B$ be the
trace-preserving conditional expectation. Explicitly,
\[
 E_B(x)=\sum_i E_ixE_i.
\]
The nonzero singular values of $e^{tA}Y$ lie in $[e^{-t},e^t]$.
The polar map is $b^{-1}$-Lipschitz in Hilbert--Schmidt norm on
matrices whose nonzero singular values are at least $b$:
apply the Hilbert--Schmidt Lipschitz estimate to the clipped sign
function on their self-adjoint dilations.
Since $A$ commutes with every $v_j$, it follows that
\begin{equation}\label{pov:scaled-energy}
 \mathcal I_{u,v}(Y_A)\le e^{4t}\mathcal I_{u,v}(Y).
\end{equation}

Put
$m_i=\tau(E_ip_A)-\tau((E_ip_AE_i)^2)\ge0$.
Differentiating the projection onto $e^{tA}YH$ gives
\[
 \partial_{s_i}p_A
 =t\bigl((1-p_A)E_ip_A+p_AE_i(1-p_A)\bigr),\qquad
 \partial_{s_i}\tau(E_ip_A)=2tm_i.
\]
Integration by parts for the uniform distribution,
$\mathbb E[af(a)]=\frac12\mathbb E[(1-a^2)f'(a)]$, therefore yields
\[
 t\,\mathbb E\sum_i(1-s_i^2)m_i
 =\mathbb E\tau(Ap_A)\le\tau(Y^*Y)\le1.
\]
For $0<\varepsilon\le1$, the terms with $|s_i|\le1-\varepsilon$
have weight $1-s_i^2\ge\varepsilon$. The other terms have expected
sum at most $\varepsilon\sum_i\tau(E_i)\le3\varepsilon$, since
$m_i\le\tau(E_i)$. Thus
\[
 \mathbb E\|p_A-E_B(p_A)\|_2^2
 =\mathbb E\sum_i m_i
 \le(t\varepsilon)^{-1}+3\varepsilon.
\]
Taking $\varepsilon=t^{-1/2}$, choose $A$ with
$D:=\|p_A-E_B(p_A)\|_2^2\le4t^{-1/2}$.

Set $a=E_B(p_A)$ and $e_0=1_{[1/2,1]}(a)$.
Since $\tau(a-a^2)=D$,
$\|p_A-e_0\|_2^2\le2D$ and
$|\tau(e_0)-\tau(p_A)|\le2D$.
Modify $e_0$ inside the blocks by adding or deleting the rank
difference, and denote the resulting block-diagonal projection by $e$.
Then $e$ and $p_A$ have the same rank.
The trace formula for projection distance gives
$\|e-p_A\|_2^2\le4D$.
Complete the polar part of $eY_A$ to an isometry from $Y^*YH$ onto
$eK$, and denote the resulting partial isometry by $Z$.
If $c_k\in[0,1]$ are the singular values of $eY_A$ on $Y^*YH$,
then
\[
 \|Z-Y_A\|_2^2
 =\frac2d\sum_k(1-c_k)
 \le\frac2d\sum_k(1-c_k^2)
 =\|e-p_A\|_2^2\le16t^{-1/2}.
\]
Now $ZZ^*=e$ and $Z^*Z=Y^*Y$.
Using $\mathcal I_{u,v}(X)\le\|X\|_2^2$ and
\eqref{pov:scaled-energy} proves \eqref{pov:finite-selection}.
\end{proof}

The next proposition is one of the main technical preparatory results. Informally, it says that we can replace a covering projection family $\boldsymbol q$ with small boundary energy with a subequivalent orthogonal one, so that the partial isometries realizing the subequivalence still have small boundary energy -- which also ensures small boundary for the new orthogonal family.

\begin{proposition}
\label{pov:rounding}
Let $u_1,\ldots,u_h\in U(d)$ and let $\boldsymbol q=(q_i)_i$
be projections with $\sum_i\tau(q_i)\le3$ and
$\sum_i\mathcal E_u(q_i)\le s$.
There are partial isometries $V_i$ with mutually orthogonal initial
projections $P_i=V_i^*V_i$ and target projections $Q_i=V_iV_i^*$ such that
$Q_i\le q_i$ and, for an absolute constant $C$,
\begin{equation}\label{pov:prescribed-ranges}
 \sum_i\mathcal E_u(V_i)
 \le C(\log(1/s))^{-1/2},\qquad
 \sum_i\mathcal E_u(P_i)
 \le C(\log(1/s))^{-1/2}.
\end{equation}
Moreover,
\begin{equation}\label{pov:defect-increase}
 \delta(\boldsymbol P)-\delta(\boldsymbol q)
 \le C(\log(1/s))^{-1/2}.
\end{equation}
\end{proposition}

\begin{proof}
First suppose $s=0$. Put
\[
 r_0=0,\qquad r_i=q_1\vee\cdots\vee q_i,\qquad
 P_i=r_i-r_{i-1}.
\]
These projections belong to $\{u_1,\ldots,u_h\}'$.
The operator $q_iP_i$ has initial support $P_i$:
a vector in $P_iH$ annihilated by $q_i$ is orthogonal both
to $r_{i-1}H$ and to $q_iH$, hence to $r_iH$, and must vanish.
Let $V_i$ be its polar part. Then
$V_i^*V_i=P_i$, $V_iV_i^*\le q_i$, and $V_i$ belongs to
the same commutant. All the required energies therefore
vanish. Moreover,
\[
 \sum_iP_i=\supp\sum_iq_i,\qquad
 \delta(\boldsymbol P)
 =\tau\left(1-\supp\sum_iq_i\right)
 \le\delta(\boldsymbol q).
\]
This proves the assertion for $s=0$.

Assume henceforth that $s>0$.
For a projection $q$ and a unitary $u$, complete the polar
part of $quq$ to a unitary $v$ on $qH$.
Its singular values $\sigma_k$ belong to $[0,1]$, so
\[
 \|v-quq\|_2^2
 =\frac1d\sum_k(1-\sigma_k)^2
 \le\frac1d\sum_k(1-\sigma_k^2)
 =\|(1-q)uq\|_2^2.
\]
The two off-diagonal blocks of a unitary have equal
Hilbert--Schmidt norm:
\[
 \|(1-q)uq\|_2^2=\|qu(1-q)\|_2^2
 =\tfrac12\|[u,q]\|_2^2.
\]
Splitting the source or target into $qH$ and $(1-q)H$
therefore gives
\begin{equation}\label{eq:source-p3}
 \|vq-qu\|_2^2\le\|[u,q]\|_2^2,\qquad
 \|uq-vq\|_2^2\le\|[u,q]\|_2^2.
\end{equation}

Make these completions $v_{j,i}$ for every $q_iu_jq_i$,
and put
\[
 K=\bigoplus_iq_iH,\qquad
 v_j=\bigoplus_i v_{j,i},\qquad
 F:H\longrightarrow K,\quad F\xi=(q_i\xi)_i.
\]
Then $\dim K=d\sum_i\tau(q_i)\le3d$.
The first estimate in \eqref{eq:source-p3} gives
\[
 \mathcal I_{u,v}(F)
 =\frac1{4h}\sum_{j,i}\|v_{j,i}q_i-q_iu_j\|_2^2
 \le\sum_i\mathcal E_u(q_i)\le s.
\]
Clip the singular values of $F$ at one and denote the
resulting contraction by $C$.
Clipping is contractive in Hilbert--Schmidt norm and
commutes with left and right unitary multiplication.
Consequently,
\[
 \mathcal I_{u,v}(C)\le s,\qquad
 |C|=\min\left\{\Bigl(\sum_iq_i\Bigr)^{1/2},1\right\},
 \qquad
 \tau((1-|C|)^2)=\delta(\boldsymbol q).
\]

Write $C=U|C|$ and, for $b>0$, set
$
 Y_b:=U1_{[b,1]}(|C|).
$
We claim that
\begin{equation}\label{pov:polar-coarea}
 \int_0^1\mathcal I_{u,v}(Y_b)\,db
 \le2\sqrt{\mathcal I_{u,v}(C)}.
\end{equation}
To prove this, consider the self-adjoint dilation and
block-diagonal unitaries
\[
 A=\begin{pmatrix}0&C^*\\C&0\end{pmatrix},
 \qquad
 W_j=\begin{pmatrix}u_j&0\\0&v_j\end{pmatrix}
\]
on $H\oplus K$.
For this calculation use the trace $d^{-1}\Tr$ and put
\[
 \widetilde{\mathcal E}(B)
 =\frac1{4hd}\sum_j\|[W_j,B]\|_{\rm HS}^2.
\]
The two off-diagonal corners of $[W_j,A]$ have equal
Hilbert--Schmidt norm, so
\(
 \widetilde{\mathcal E}(A)=2\mathcal I_{u,v}(C).
\)
Put $p_s=1_{[s,\infty)}(A)$. For almost every $b\in(0,1)$,
functional calculus gives
\[
 p_b+p_{-b}-I=\begin{pmatrix}0&Y_b^*\\Y_b&0\end{pmatrix}.
\]
Thus $\widetilde{\mathcal E}(p_b+p_{-b})=2\mathcal I_{u,v}(Y_b)$.
The squared triangle inequality and Lemma~\ref{pre:coarea},
applied to $A$ on $[-1,1]$, now give
\[
  \int_0^1\mathcal I_{u,v}(Y_b)\,db
 \le\int_{-1}^1\widetilde{\mathcal E}(p_s)\,ds\\
 \le\frac12\sqrt{\frac{d+\dim K}{d}}\,
       \widetilde{\mathcal E}(A)^{1/2}
 \le\sqrt{2\mathcal I_{u,v}(C)}.
\]
The dimension factor accounts for rescaling the normalized trace
on $H\oplus K$ to $d^{-1}\Tr$; the last inequality uses $\dim K\le3d$.
This proves \eqref{pov:polar-coarea}.

Fix $0<a\le1/4$. Averaging over $[a,2a]$ gives some
$b\in[a,2a]$ such that $Y=Y_b$ satisfies
\[
 \mathcal I_{u,v}(Y)\le2\sqrt{s}/a.
\]
Also,
\[
 \begin{aligned}
 \tau(1-Y^*Y)
 &=\tau(1_{[0,b)}(|C|))\le\tau((1-|C|)^2)+2b
 \le\delta(\boldsymbol q)+4a,
 \end{aligned}
\]
where we used
$1_{[0,b)}(x)\le(1-x)^2+2b$ on $[0,1]$.

Let $E_i$ be the projection of $K$ onto its $i$th summand.
These projections commute with every $v_j$.
For any $t\ge1$, Lemma~\ref{pov:selection} supplies a
partial isometry $Z:H\to K$ such that
\[
 Z^*Z=Y^*Y,\qquad [ZZ^*,E_i]=0,\qquad
 \mathcal I_{u,v}(Z)
 \le2e^{4t}\mathcal I_{u,v}(Y)+32t^{-1/2}.
\]
Identify $E_iK$ with $q_iH$, let $J_i:q_iH\to H$ be
the inclusion, and put
\[
 V_i=J_iE_iZ,\qquad R=ZZ^*.
\]
Because $R$ commutes with the $E_i$, the operators
\[
 P_i=V_i^*V_i=Z^*E_iZ,\qquad
 Q_i=V_iV_i^*=J_iE_iRE_iJ_i^*
\]
are projections, and $Q_i\le q_i$.
Indeed,
\[
 P_iP_k=Z^*E_iRE_kZ
 =\begin{cases}P_i,&i=k,\\0,&i\ne k,\end{cases}
 \qquad
 \sum_iP_i=Z^*Z=Y^*Y.
\]
Thus the $V_i$ have mutually orthogonal initial projections,
and
\[
 \delta(\boldsymbol P)=\tau(1-Y^*Y)
 \le\delta(\boldsymbol q)+4a.
\]

For the energy estimate, write
\[
 u_jV_i-V_iu_j
 =(u_jJ_i-J_iv_{j,i})E_iZ
       +J_iE_i(v_jZ-Zu_j).
\]
The second estimate in \eqref{eq:source-p3} and
$\|E_iZ\|\le1$ give
\[
 \sum_i\|(u_jJ_i-J_iv_{j,i})E_iZ\|_2^2
 \le\sum_i\|[u_j,q_i]\|_2^2.
\]
Since the $E_i$ form an orthogonal partition of $K$,
\[
 \sum_i\|J_iE_i(v_jZ-Zu_j)\|_2^2
 =\|v_jZ-Zu_j\|_2^2.
\]
Using the squared triangle inequality and summing over $j$,
we obtain
\[
 \sum_i\mathcal E_u(V_i)
 \le2s+2\mathcal I_{u,v}(Z).
\]
Consequently,
\begin{equation}\label{pov:orthogonalization-free-scales}
 \begin{aligned}
 \sum_i\mathcal E_u(V_i)
 &\le2s+8e^{4t}\sqrt{s}/a+64t^{-1/2} \mbox{ and }
 \delta(\boldsymbol P)-\delta(\boldsymbol q)&\le4a.
 \end{aligned}
\end{equation}

Choose
\(
 t=\tfrac1{16}\log(1/s),\, a=s^{1/8}.
\)
For sufficiently small $s$, these parameters are admissible,
and
\[
 e^{4t}\sqrt{s}/a=s^{1/8},\qquad
 t^{-1/2}=4(\log(1/s))^{-1/2}.
\]
Both bounds in \eqref{pov:orthogonalization-free-scales}
are therefore $O((\log(1/s))^{-1/2})$.
Finally, the product rule and $\|V_i\|\le1$ give
\[
 \|[u_j,V_i^*V_i]\|_2
 \le\|[u_j,V_i^*]\|_2+\|[u_j,V_i]\|_2
 =2\|[u_j,V_i]\|_2.
\]
Thus $\mathcal E_u(P_i)\le4\mathcal E_u(V_i)$, proving
the remaining assertion.
\end{proof}

We now distill the results of this section into one proposition which will be used in the proof of the spectral gap theorem. There it will ensure that we can correct an inductively constructed almost invariant family $\boldsymbol f$ with small boundary energy to an orthogonal one.

\begin{proposition}\label{phy:summary}
Let $u=(u_1,\ldots,u_h)\in U(d)^h$ and let
$e,f_1,\ldots,f_m\in M_d$ be projections.
Put $S_0=1-e$ and $S_i=S_0+\sum_{j\le i}f_j$.
For sufficiently small $\eta>0$, suppose
\[
 \begin{gathered}
 \tau(1-e)+\sum_i\tau(f_i)\le3,\qquad
 \mathcal E_u(1-e)+\sum_i\mathcal E_u(f_i)\le e^{-\eta^{-4}},\\
 \tau(1-e)+\tau(1_{[0,\eta]}(S_m))\le\eta.
 \end{gathered}
\]
There are partial isometries $V_i\in M_d$ with mutually orthogonal
initial projections $P_i=V_i^*V_i$ and target projections $Q_i=V_iV_i^*$
such that
\begin{equation}\label{phy:partial-isometry-estimates}
 \rank P_i\le\rank f_i,\qquad
 \sum_i\mathcal E_u(V_i)=O(\eta^2),\qquad
 \sum_i\tau(S_{i-1}^2Q_i)\le3\eta,
\end{equation}
and, for $P_0=1-\sum_iP_i$,
\begin{equation}\label{phy:orthogonal-covering}
 \tau(P_0)=O(\eta).
\end{equation}
\end{proposition}
\begin{proof}
Apply Lemma~\ref{phy:covariance} with $\gamma = e^{-\eta^{-4}/16}$. The resulting projections $q_i$ have total energy
$O(e^{-\eta^{-4}/8})$ and satisfy
$\sum_i\tau(S_{i-1}^2q_i)\le3\eta$.
Since $e^{-\eta^{-4}/16}\le\eta$ for sufficiently small $\eta$,
monotonicity of the low spectral projections and the last hypothesis
give $\delta(\boldsymbol q)=O(\eta)$.
Proposition~\ref{pov:rounding} supplies $Q_i\le q_i$ with energy
and additional coverage defect $O(\eta^2)$.
The rank and weighted trace bounds pass to $Q_i$, proving the assertions.
\end{proof}

\section{Expanding decomposition with a uniform gap}
\sectionmark{Expanding decomposition with a uniform gap}
\label{phy:constructionsection}

\begin{definition}
We say that a unitary tuple
$v=(v_1,\ldots,v_k)$ on a nonzero corner $R$ has scalar gap $c$ if
\begin{equation}\label{phy:scalar-gap}
 \frac1{4k}\sum_j\|[v_j,x]\|_2^2
 \ge c\left\|x-\frac{\tau(x)}{\tau(R)}R\right\|_2^2,
 \qquad x\in RM_dR.
\end{equation}
Here the trace and the $2$-norm are those of the ambient matrix algebra.
\end{definition}
This section is devoted to the proof of the spectral gap theorem which will yield a partition of the identity $(e_i)$ such that the corrected doubled unitary tuple has scalar gap at least $\kappa^2/2^{28}$ on every corner of the partition.

Throughout this section, constants implicit in $O(\cdot)$ may depend
only on $\kappa$ and $h$. 
\subsection{From almost expansion to the spectral gap}

The next preliminary lemma corrects a projection family which has ``almost expansion'' with respect to a family of unitaries to the one with genuine spectral gap with respect to a small perturbation of doubles of these unitaries.

\begin{lemma}\label{phy:pruning}
Let $0<\kappa'<1$ and let $(P_0,\ldots,P_m)$ be a projection partition reduced by
$w_1,\ldots,w_h\in U(d)$. Suppose $\xi_i\ge0$ for $1\le i\le m$
and that, for every $1\le i\le m$ and every projection $q\le P_i$
with $\rank q\le\rank P_i/2$,
\begin{equation}\label{eq:source-g1}
 \mathcal E_w(q)\ge\kappa'\tau(q)-\xi_i.
\end{equation}
Put $\xi=\sum_{i\ge1}\xi_i$. There are $e_i\le P_i$, $i\ge1$,
and unitaries $v_j^\pm$ reducing the partition
$e_0=1-\sum_{i\ge1}e_i,e_1,\ldots,e_m$, such that
\[
 \sum_{i\ge1}\mathcal E_w(e_i)\le\xi/8,
 \qquad
 \tau(e_0)\le\tau(P_0)+\frac{16\xi}{\kappa'},
\]
and
\[
 \sum_{j,\pm}\|v_j^\pm-w_j\|_2^2
 \le2h\xi+8h\tau(e_0)
 \le2h\xi+8h\tau(P_0)+\frac{128h\xi}{\kappa'}.
\]
The doubled tuple has scalar gap $(\kappa')^2/256$ on every nonzero
$e_i$, $i\ge1$.
\end{lemma}
\begin{proof}
Fix $P=P_i\ne0$. Discard $P$ if $\xi_i>\kappa'\tau(P)/16$.
Otherwise start with $Q=P$ and repeatedly remove a nonzero
$p\le Q$ with $\tau(p)\le\tau(Q)/2$ and
\[
 \mathcal E_{w,Q}(p):=(4h)^{-1}\sum_j\|[Qw_jQ,p]\|_2^2
 <(\kappa'/16)\tau(p).
\]
Ranks decrease without exhausting $Q$, so the process ends at $e\ne0$.
For an initial sum $F$ of removed projections, orthogonality of the
matrix entries joining each removed projection to the final complement gives
\begin{equation}\label{eq:source-g2}
 \mathcal E_w(F)=\mathcal E_w(P-F)
 \le\sum_k\mathcal E_{w,Q_k}(p_k)\le(\kappa'/16)\tau(F).
\end{equation}
At a first crossing of $\tau(P)/4$, one would have
$\tau(P)/4\le\tau(F)<5\tau(P)/8$. The smaller of $F,P-F$ then has
trace in $[\tau(P)/4,\tau(P)/2]$. Its energy is at least
$\kappa'\tau(P)/4-\xi_i\ge3\kappa'\tau(P)/16$ by \eqref{eq:source-g1},
but at most $5\kappa'\tau(P)/128$ by \eqref{eq:source-g2}, a contradiction.
Thus the final removed part has trace less than $\tau(P)/4$, and
\[
 \tau(P-e)\le2\xi_i/\kappa',\qquad \mathcal E_w(e)\le\xi_i/8.
\]
Thus a retained block contributes at most $2\xi_i/\kappa'$ to $e_0$,
whereas a wholly discarded block satisfies
$\tau(P_i)<16\xi_i/\kappa'$. Consequently,
\[
 \tau(e_0)=\tau(P_0)+\sum_{i\ge1}\tau(P_i-e_i)
 \le\tau(P_0)+\frac{16\xi}{\kappa'}.
\]
The energy bound follows by summing over the retained blocks.
Every $q\le e$ of rank at most half that of $e$ satisfies
$\mathcal E_{w,e}(q)\ge(\kappa'/16)\tau(q)$.

For each retained nonzero block $e=e_i$, $i\geq 1$, we perform the following construction. Complete the polar part of $b=ew_je$ to a unitary $v$ and set
$v_{j,e}^{\pm}=v(|b|\pm i(e-|b|^2)^{1/2})$.
These unitaries average to $b$; sum them over $e$ to define $v_j^\pm$.
The trace-preserving u.c.p. channel
\[
 \mathcal H_e=(4h)^{-1}\sum_{j,\pm}
       \bigl(\Ad v_{j,e}^{\pm}+\Ad((v_{j,e}^{\pm})^*)\bigr)
\]
satisfies, using the normalized trace on each retained corner, for every at most half-rank projection $q$,
\[
 \tau((e-q)\mathcal H_e(q))
 =(4h)^{-1}\sum_{j,\pm}\|[v_{j,e}^{\pm},q]\|_2^2
 \ge(\kappa'/8)\tau(q).
\]
The quantum Cheeger inequality~\cite[Theorem~1.9(2)]{LiQiaoWigderson}
gives gap at least $(\kappa')^2/128$ for $I-\mathcal H_e$. The energy in
\eqref{phy:scalar-gap} is half this form, giving $(\kappa')^2/256$; rank-one
blocks satisfy the inequality automatically. On $e_0$ we set
$v_j^{\pm}=e_0$.

Returning to the ambient trace, write
\[
 b_{j,i}=e_iw_je_i,
 \qquad
 C_j=\sum_{i\ge0}b_{j,i}.
\]
Orthogonality of the matrix corners gives
\begin{align*}
 \|v_j^\pm-w_j\|_2^2
 &=\|w_j-C_j\|_2^2
   +\sum_{i\ge1}\|v_{j,e_i}^\pm-b_{j,i}\|_2^2
   +\|e_0-b_{j,0}\|_2^2 \le2\|w_j-C_j\|_2^2+4\tau(e_0).
\end{align*}
Moreover, because each $P_i$ reduces the original tuple,
\[
 \mathcal E_w(e_0)=\sum_{i\ge1}\mathcal E_w(e_i).
\]
Consequently,
\begin{align*}
 \sum_{j,\pm}\|v_j^\pm-w_j\|_2^2
 &\le8h\sum_{i\ge0}\mathcal E_w(e_i)+8h\tau(e_0)=16h\sum_{i\ge1}\mathcal E_w(e_i)+8h\tau(e_0)\\
 &\le2h\xi+8h\tau(e_0)\le2h\xi+8h\tau(P_0)+\frac{128h\xi}{\kappa'}.
\end{align*}
\end{proof}

We have now assembled all preparatory material for our spectral gap theorem.

\begin{theorem}[Theorem~\ref{main:correction}]\label{phy:decomposition}
Suppose the tuple $(u_1,\dots,u_h)$ has spectral gap in the ultraproduct. There are projection partitions
$(e_{0,n},\ldots,e_{m_n,n})$ and unitaries $u_{j,n}^{\pm}$ in the
original matrices such that
\begin{equation}\label{phy:vanishing-corner}
 [u_{j,n}^{\pm},e_{i,n}]=0,\qquad
 \sum_{j,\pm}\|u_{j,n}^{\pm}-u_{j,n}\|_2^2
 \longrightarrow_\omega0
\end{equation}
and on every nonzero $e_{i,n}$ the doubled tuple has scalar
gap at least $\kappa^2/2^{28}$.
\end{theorem}
\begin{proof}
We first fix notation for the control quantities. Let $(r_n,\alpha_n)$ be the data of Theorem~\ref{pre:repair}. For sufficiently small $\alpha_n$, put
\begin{equation}\label{phy:repair-scale}
 \eta_n=(\log(1/\alpha_n))^{-1/8},
\end{equation}
and put $\eta_n=1$ elsewhere. Then $\eta_n\to_\omega 0$, and
$\alpha_n\le\eta_n^4$, $9\alpha_n\le e^{-\eta_n^{-4}}$ and $\tau(1-r_n)\le \alpha_n$ on an $\omega$-large set. We will suppress $n$ in what follows; all limits are taken along $\omega$.

\smallskip\noindent\textbf{Step 1: selecting small boundary projections and correcting them to almost-invariant ones.}
Start with $S_0=1-r$. At step $i$, put
$b_i=1_{[0,\kappa/512]}(S_{i-1})$ and let $w_i$ be the polar
part of $rb_i$.
Since $S_{i-1}\ge1-r$, we have
$b_i(1-r)b_i\le(\kappa/512)b_i$.
Thus $w_i^*w_i=b_i$, $w_iw_i^*\le r$, and
$b_iw_ib_i=(b_irb_i)^{1/2}$ gives, for every projection $p\le b_i$,
\begin{equation}\label{phy:corner-transfer}
 \|w_ipw_i^*-p\|_2^2\le\frac\kappa{256}\tau(p),\qquad
 \mathcal E(w_ipw_i^*)\le2\mathcal E(p)+\frac\kappa{128}\tau(p).
\end{equation}
Choose a nonzero $p_i\le w_iw_i^*$ of minimum rank satisfying
\begin{equation}\label{phy:admissible}
 \mathcal E(p_i)<\frac\kappa{64}\tau(p_i),
\end{equation}
and stop when none exists. Theorem~\ref{pre:repair} gives a projection $f_i$ with
\begin{equation}\label{phy:repair}
 \|f_i-p_i\|_2^2\le\tfrac9{16}\tau(p_i),\qquad
 \mathcal E(f_i)\le\alpha\tau(p_i),\qquad
 \tfrac13\rank p_i\le\rank f_i\le\tfrac53\rank p_i.
\end{equation}
Put $S_i=S_{i-1}+f_i$ and $\widetilde p_i=w_i^*p_iw_i\le b_i$.
By \eqref{phy:corner-transfer}--\eqref{phy:repair},
\[
 \|f_i-\widetilde p_i\|_2^2
 \le\left(\tfrac34+\tfrac1{16}\right)^2\tau(p_i)
 <\tfrac34\tau(p_i).
\]
The projection-distance identity therefore gives
\[
 \tau(b_if_i)\ge\tau(\widetilde p_if_i)
 =\tfrac12\bigl(\tau(p_i)+\tau(f_i)
                  -\|f_i-\widetilde p_i\|_2^2\bigr)
 \ge\tfrac12\tau(f_i).
\]
To bound the total trace, set
$\psi(S)=\tau(1-(1+4S)^{-1})$ for $S\ge0$ and
$D_i=(1+4S_{i-1})^{-1}$. The resolvent identity gives
\[
 \begin{aligned}
 \psi(S_i)-\psi(S_{i-1})
 &=4\tau\!\left(D_if_i(1+4f_iD_if_i)^{-1}f_iD_i\right)\ge\tfrac45\tau(f_iD_i^2)
 \ge\frac{4\tau(b_if_i)}{5(1+\kappa/128)^2}
 \ge\tfrac38\tau(f_i).
 \end{aligned}
\]
Here $0\le D_i\le1$ and
$D_i^2\ge(1+\kappa/128)^{-2}b_i$.
Since $0\le\psi\le1$ and each step increases it by at least
$\tau(p_i)/8\ge1/(8d)$, the construction terminates, with
\begin{equation}\label{phy:accounting}
 \sum_i\tau(p_i)\le8,\qquad
 \sum_i\tau(f_i)\le8/3,\qquad
  \tau(1-r)+\sum_i\mathcal E(f_i)\le9\alpha\le e^{-\eta^{-4}}.
\end{equation}
At this point our projection family $\boldsymbol f$ has small boundary energy, but potentially bad coverage (it can be very non-orthogonal). In the next step, we correct that.

\smallskip\noindent\textbf{Step 2: controlling coverage and orthogonalizing.}
The input family $1-r,f_1,\ldots,f_m$ has total trace at most
three; its energy satisfies the bound arranged before the construction.
The Araki--Yamagami inequality, applied to its column operator
using the unitary completions in \eqref{eq:source-p3}, gives
\[
 \mathcal E(S_m^{1/2})
 \le2\left(\mathcal E(1-r)+\sum_i\mathcal E(f_i)\right)
 \le18\eta^4.
\]
Lemma~\ref{pre:coarea} supplies
$q=1_{[0,\theta]}(S_m)\le b_{m+1}$, with
$\theta\in[\kappa/1024,\kappa/512]$ and
$\mathcal E(q)=O(\eta^2)$.
Stopping and \eqref{phy:corner-transfer} give
$\mathcal E(q)\ge\kappa\tau(q)/256$, so $\tau(q)=O(\eta^2)$.
For sufficiently small $\eta$, we have $\eta<\kappa/1024$, hence
\[
 \tau(1-r)+\tau(1_{[0,\eta]}(S_m))=O(\eta^2)\le\eta.
\]
Moreover, we have
\[
 \mathcal E(1-r)+\sum_i\mathcal E(f_i)
 \le\tau(1-r)+\alpha\sum_i\tau(p_i)
 \le9\alpha
 \le e^{-\eta^{-4}}.
\]
Proposition~\ref{phy:summary} now gives partial isometries $V_i$ with
orthogonal initial projections $P_i=V_i^*V_i$ and target projections
$Q_i=V_iV_i^*$ satisfying
\begin{equation}\label{phy:partial-isometries}
 \rank P_i\le\rank f_i<2\rank p_i,\qquad
 \sum_i\mathcal E(V_i)=O(\eta^2),\qquad
 \tau(P_0)=O(\eta),
\end{equation}
where $P_0=1-\sum_iP_i$.
Moreover $1-b_i\le(512/\kappa)^2S_{i-1}^2$, so
\begin{equation}\label{phy:range-defect}
 \sum_i\tau((1-b_i)Q_i)=O(\eta).
\end{equation}

\smallskip\noindent\textbf{Step 3: ensuring uniform almost-expansion beneath $P_i$.}
For every unitary $u$, block orthogonality gives
\[
 \|[u,P_0]\|_2^2
 =\sum_{i\ge1}
   \bigl(\|P_0uP_i\|_2^2+\|P_iuP_0\|_2^2\bigr)
 \le\sum_{i\ge1}\|[u,P_i]\|_2^2.
\]
Together with $\mathcal E(P_i)\le4\mathcal E(V_i)$,
this implies
\[
 \sum_{i\ge0}\mathcal E(P_i)
 \le2\sum_{i\ge1}\mathcal E(P_i)
 \le8\sum_{i\ge1}\mathcal E(V_i).
\]
Put $a_j=\sum_{i\ge0}P_iu_jP_i$, and complete the polar
part of each nonzero diagonal block to a unitary on that
block. Their direct sum is a unitary $w_j$ reducing the
partition.
The singular-value estimate used in the proof of
Proposition~\ref{pov:rounding} gives
\[
 \|w_j-a_j\|_2^2
 \le\sum_{i\ge0}\|(1-P_i)u_jP_i\|_2^2
 =\|u_j-a_j\|_2^2.
\]
Since $w_j-a_j$ is block diagonal and $u_j-a_j$ is
off diagonal, these operators are orthogonal in $L^2$.
Thus
\[
 \begin{aligned}
 \|w_j-u_j\|_2^2
 &=\|w_j-a_j\|_2^2+\|u_j-a_j\|_2^2\le2\|u_j-a_j\|_2^2
 =\sum_{i\ge0}\|[u_j,P_i]\|_2^2.
 \end{aligned}
\]
Summing over $j$, we conclude that
\[
 \sum_j\|w_j-u_j\|_2^2
 \le4h\sum_{i\ge0}\mathcal E(P_i)
 \le32h\sum_{i\ge1}\mathcal E(V_i)
 =O(\eta^2).
\]
Put $\Delta_{j,i}^{\pm}=u_j^{\pm1}V_i-V_iw_j^{\pm1}$ and define
\begin{equation}\label{phy:additive-defect}
 \xi_i=\frac1{2h}\sum_{j,\pm}\|\Delta_{j,i}^{\pm}\|_2^2
             +2\tau((1-b_i)Q_i).
\end{equation}
To estimate the total defect, write
\[
 \Delta_{j,i}^{\pm}
 =[u_j^{\pm1},V_i]+V_i(u_j^{\pm1}-w_j^{\pm1}).
\]
Since $\sum_iV_i^*V_i\le1$, every matrix $x$ satisfies
\[
 \sum_i\|V_ix\|_2^2
 =\tau\left(x^*\Bigl(\sum_iV_i^*V_i\Bigr)x\right)
 \le\|x\|_2^2.
\]
Moreover, unitary invariance of the $2$-norm gives
\[
 \|[u_j^{-1},V_i]\|_2=\|[u_j,V_i]\|_2,\qquad
 \|u_j^{-1}-w_j^{-1}\|_2=\|u_j-w_j\|_2.
\]
The squared triangle inequality therefore yields
\[
 \sum_{i,\pm}\|\Delta_{j,i}^{\pm}\|_2^2
 \le4\sum_i\|[u_j,V_i]\|_2^2
       +4\|u_j-w_j\|_2^2.
\]
Summing over $j$ and using \eqref{phy:range-defect},
we obtain
\[
 \begin{aligned}
 \sum_i\xi_i
 &\le8\sum_i\mathcal E(V_i)
       +\frac2h\sum_j\|u_j-w_j\|_2^2
       +2\sum_i\tau((1-b_i)Q_i)\\
 &=O(\eta^2)+O(\eta)
 =O(\eta).
 \end{aligned}
\]

For a projection $q\le P_i$ of rank at most half of $\rank P_i$, put $y=V_iqV_i^*\le Q_i$ and
$\widehat y=1_{[1/2,1]}(b_i yb_i)$.
Then $\rank\widehat y\le\rank y<\rank p_i$ and
\[
 \|y-\widehat y\|_2^2\le2\tau((1-b_i)y),\qquad
 \tau(\widehat y)\ge\tau(y)-2\tau((1-b_i)y).
\]
Indeed, for $x=b_i yb_i$, scalar functional calculus gives
$2x+\widehat y-2x\widehat y\le\supp x$;
take traces and use $\tau(\widehat y y)=\tau(\widehat y x)$ and
$\rank x\le\rank y$.
Minimality applied to $w_i\widehat yw_i^*$, together with
\eqref{phy:corner-transfer}, gives
$\mathcal E(\widehat y)\ge\kappa\tau(\widehat y)/256$.
Since $\mathcal E(\widehat y)\le2\mathcal E(y)+2\|\widehat y-y\|_2^2$,
\[
 \mathcal E(y)\ge\frac\kappa{512}\tau(q)
                         -3\tau((1-b_i)Q_i).
\]
The identity
\[
 [u_j,y]=V_i[w_j,q]V_i^*
       +\Delta_{j,i}^+qV_i^*-V_iq(\Delta_{j,i}^-)^*
\]
now gives
\begin{equation}\label{phy:almost-expansion}
 \mathcal E_w(q)\ge\frac\kappa{2^{10}}\tau(q)-\xi_i,
 \qquad \sum_i\xi_i=O(\eta).
\end{equation}
Apply Lemma~\ref{phy:pruning} with $\kappa'=\kappa/2^{10}$.
The retained blocks have scalar gap
$(\kappa')^2/256=\kappa^2/2^{28}$. Restoring the coordinate index and
putting $\xi_n=\sum_i\xi_{i,n}$, the trace estimate in
Lemma~\ref{phy:pruning} gives
\[
 \tau(e_{0,n})
 \le\tau(P_{0,n})+\frac{16\xi_n}{\kappa'}
 =O(\eta_n).
\]
The corrected tuple is the identity on $e_{0,n}$. Decompose this
projection into orthogonal rank-one projections, add them to the retained
blocks, and relabel the resulting partition. This refinement does not
change the unitaries. On each new rank-one corner every element is scalar,
so the right-hand side of~\eqref{phy:scalar-gap} is zero. Thus every
nonzero block of the refined partition has scalar gap at least
$\kappa^2/2^{28}$.

If $v_{j,n}^{\pm}$ denotes the tuple supplied by
Lemma~\ref{phy:pruning}, its corrected edit estimate and the preceding
construction of $w_{j,n}$ give
\begin{align*}
 \sum_{j,\pm}\|v_{j,n}^{\pm}-u_{j,n}\|_2^2
 &\le4h\xi_n+16h\tau(e_{0,n})
       +4\sum_j\|w_{j,n}-u_{j,n}\|_2^2=O(\eta_n)\longrightarrow_\omega0.
\end{align*}
Set $u_{j,n}^{\pm}=v_{j,n}^{\pm}$ on the $\omega$-large set where the
construction applies. On all remaining coordinates, take
$u_{j,n}^{\pm}=1$ and choose a rank-one projection partition. These
choices do not change the ultraproduct classes or the limit above and
make the assertion valid at every coordinate.
\end{proof}

\section{Scalar diagonal estimates and reconstruction}
\sectionmark{Scalar diagonal estimates and reconstruction}
\label{cp:section}

The purpose of this section is to construct finite-dimensional
subalgebras $A_n\subset M_{d_n}(\mathbb C)$ that realize the commutant
$C$ as an internal subalgebra of the tracial ultraproduct. The fact that a spectral gap condition implies internality had previously been used in \cite[Lemma 2.2]{PopaSpectralGap} in a different setting.

\subsection{Scalar diagonal lifts and a MASA}

We first show that a MASA of $C$ is internal.
In~\eqref{phy:scalar-lifts} and~\eqref{cp:diagonal-control}, the scalar
trace is identified with the map
\[
 x\longmapsto\tau_i(x)e_{i,n}.
\]

\begin{corollary}\label{phy:lifts}
For the partitions of Theorem~\ref{phy:decomposition}, put
$B_{i,n}=e_{i,n}M_{d_n}e_{i,n}$ and
$D_n=\bigoplus_i\mathbb Ce_{i,n}$, omitting zero corners.
There are trace-preserving u.c.p. maps $\Psi_n$ inducing $E_C$ which are
$D_n$-bimodular, self-adjoint on \(L^2(M_{d_n},\tau)\) and satisfy
\begin{equation}\label{phy:scalar-lifts}
 \max_i\|\Psi_n|_{B_{i,n}}-\tau_i\|_{2\to2}
 \longrightarrow_\omega0.
\end{equation}
Moreover $[D_n]_\omega$ is a MASA in $C$.
\end{corollary}
\begin{proof}
Let $Q_n$ be the lazy Markov operator of the doubled tuple
$(u^\pm_{j,n})_{j,\pm}$, defined as in
\eqref{pre:lazy-markov}, and let $\eta_n$ be
the sequence chosen in the proof of Theorem~\ref{phy:decomposition}.
Put $N_n=\max\{1,\lfloor\eta_n^{-1/4}\rfloor\}$.
Since $\ell_n^{-1}\le\alpha_n\le\eta_n^4$, we have
$N_n\le\ell_n$ on an $\omega$-large set, and
$N_n\to_\omega\infty$.
Proposition~\ref{pre:expectation} shows that $T_n^{N_n}$ induces $E_C$.
The edit bound and telescoping powers give
\[
 \|Q_n^{N_n}-T_n^{N_n}\|_{\infty\to2}
 \le N_n\|Q_n-T_n\|_{\infty\to2}
 =O(\eta_n^{1/4})\longrightarrow_\omega0.
\]
Define $\Psi_n=Q_n^{N_n}$. This map is a $D_n$-bimodular trace-preserving u.c.p. map; since \(Q_n\) is self-adjoint on \(L^2(M_{d_n},\tau)\), so is \(\Psi_n\). On every nonzero corner its
scalar error is at most
$(1-\kappa^2/2^{28})^{N_n}$ by Theorem~\ref{phy:decomposition}.
This proves \eqref{phy:scalar-lifts}.

For completeness, write $D=[D_n]_\omega$,
$\mathcal B_n=\bigoplus_iB_{i,n}$ and
$\mathcal P_n(x)=\sum_ie_{i,n}xe_{i,n}$.
For independent uniform signs $\varepsilon_i\in\{-1,1\}$, block
orthogonality gives the exact identity
\begin{equation}\label{cp:independent-signs}
 \mathbb E_\varepsilon
 \left\|\left[x,\sum_i\varepsilon_ie_{i,n}\right]\right\|_2^2
 =2\sum_{i\ne j}\|e_{i,n}xe_{j,n}\|_2^2
 =2\|x-\mathcal P_n(x)\|_2^2.
\end{equation}
If $(x_n)_\omega$ commutes with $D$, choose maximizing signs in
each coordinate. The resulting unitaries belong to $D_n$, so the
left side tends to zero and
$\|x_n-\mathcal P_n(x_n)\|_2\to_\omega0$.
The converse inclusion is immediate; hence
$D'\cap\M=[\mathcal B_n]_\omega$.
On this algebra \eqref{phy:scalar-lifts}, summed over the orthogonal
corners, gives $E_C=E_D$.
Unital bimodularity gives $D\subset C$, and consequently
$D'\cap C=D$. Thus $D$ is a MASA in $C$.
\end{proof}

\subsection{Reconstruction}

We now construct subalgebras $A_n\subset M_{d_n}$ satisfying
$[A_n]_\omega=C$.

\begin{theorem}\label{cp:rounding}
Let $N\subset\M$ be a unital von Neumann subalgebra.
Let $(e_{i,n})_i$ be a partition of $1$ into nonzero projections, and put
$B_{i,n}=e_{i,n}M_{d_n}e_{i,n}$,
$D_n=\bigoplus_i\mathbb Ce_{i,n}$ and $m_{i,n}=\rank e_{i,n}$.
Suppose $E_N$ is induced by trace-preserving u.c.p. maps $F_n$, self-adjoint
on $L^2(M_{d_n},\tau)$, which are $D_n$-bimodular and satisfy
\begin{equation}\label{cp:diagonal-control}
 \sigma_n:=\max_i\|F_{ii,n}-\tau_i\|_{2\to2}\to_\omega0,
 \qquad F_{ij,n}=F_n|_{e_{i,n}M_{d_n}e_{j,n}}.
\end{equation}
Then there are unital $*$-subalgebras $A_n\supset D_n$ such that
\[
 \|F_n-E_{A_n}\|_{\infty\to2}\to_\omega0,
 \qquad N=[A_n]_\omega,\qquad Z(N)=[Z(A_n)]_\omega.
\]
\end{theorem}

\begin{proof}
By Lemma~\ref{pre:transfer},
\[
 \delta_n:=\|F_n^2-F_n\|_{\infty\to2}\longrightarrow_\omega0,
\]
because the induced map is $E_N^2-E_N=0$.
Put
\[
 \rho_n=(\delta_n+\sigma_n+1/n)^{1/4}.
\]
Then $\rho_n\to_\omega0$ and
$\delta_n,\sigma_n\le\rho_n^4$.
We perform the construction on an $\omega$-large set where
$\rho_n$ is smaller than a sufficiently small absolute
constant. On the remaining coordinates take $A_n=M_{d_n}$.
Suppress $n$ throughout the construction; all implicit
constants below are absolute.

Since $F$ is unital, completely positive and trace-preserving,
Schwarz's inequality shows that it is a Hilbert--Schmidt
contraction. By $D$-bimodularity it preserves every rectangle
$e_iM_de_j$. Each restriction $F_{ij}$ is self-adjoint on
that Hilbert space, and its spectrum lies in $[-1,1]$.

\smallskip\noindent
\textbf{Spectral analysis of rectangular blocks.}
Let $F_{ij}(v)=\lambda v$, and normalize
$\|v\|_{\mathrm{HS}}^2=M$, where
\[
 M=\max(m_i,m_j),\qquad t=\min(m_i,m_j).
\]
Suppose first that $m_j=M$ and put $a=v^*v$.
Schwarz's inequality gives
\[
 \lambda^2a=F(v)^*F(v)\le F_{jj}(a).
\]
Multiplying by $a$ and taking the trace, we obtain
\[
 \begin{aligned}
 \lambda^2\Tr(a^2)
 &\le\Tr(aF_{jj}(a))=\frac{(\Tr a)^2}{m_j}
   +\Tr\bigl(a(F_{jj}(a)-\tau_j(a)e_j)\bigr)\le M+\sigma\Tr(a^2).
 \end{aligned}
\]
If $m_i=M$, the same argument uses $vv^*$ in the $i$-th
corner. Since $\rank(v)\le t$, Cauchy--Schwarz gives
$\Tr|v|^4\ge M^2/t$. Hence
\begin{equation}\label{cp:fourth-moment}
 (\lambda^2-\sigma)\Tr|v|^4\le M,\qquad
 \lambda^2-\sigma\le t/M.
\end{equation}

Suppose that \( |\lambda|>\rho, 1-\lambda>\rho.\)
For sufficiently small $\rho$, the bound $\sigma\le\rho^4$
implies $\lambda^2-\sigma\ge\lambda^2/2>0$.
Set
\[
 R=\frac{2}{\sqrt{\lambda^2-\sigma}},
\]
and let $v_R$ be obtained by clipping the singular values of
$v$ at $R$.
The scalar inequality $(s-R)_+^2\le s^4/R^2$ gives
\[
 \|v-v_R\|_{\mathrm{HS}}^2
 \le R^{-2}\Tr|v|^4\le M/4.
\]
Therefore
\[
 \langle v,v_R\rangle
 \ge M-\|v\|_{\mathrm{HS}}\|v-v_R\|_{\mathrm{HS}}
 \ge M/2.
\]
The matrix $x=v_R/R$ is a contraction.
Self-adjointness of $F_{ij}$ gives
\[
 \begin{aligned}
 \|(F_{ij}^2-F_{ij})x\|_{\mathrm{HS}}
 &\ge
 \frac{|\langle v,(F_{ij}^2-F_{ij})x\rangle|}{\sqrt M}=\frac{|\lambda|(1-\lambda)
                  \langle v,v_R\rangle}{R\sqrt M}\ge\frac{|\lambda|(1-\lambda)\sqrt M}{2R}.
 \end{aligned}
\]
Squaring and using $\lambda^2-\sigma\ge\lambda^2/2$,
we obtain
\[
 \|(F_{ij}^2-F_{ij})x\|_{\mathrm{HS}}^2
 \ge\frac{M\lambda^4(1-\lambda)^2}{32}
 \ge\frac{M\rho^6}{32}.
\]

Form a graph on the block indices by joining distinct $i,j$
whenever $F_{ij}$ has an eigenvalue in \( |\lambda|>\rho, 1-\lambda>\rho.\). Adjoint preservation makes this condition symmetric. There are no loops in this graph: $F_{ii}(e_i)=e_i$, and the scalar diagonal estimate implies that all eigenvalues different from $1$ of $F_{ii}$ lie in $[-\sigma,\sigma]\subset[-\rho,\rho]$.

Choose a maximal matching in this graph.
For each matched pair choose the contraction just constructed,
with either orientation of its rectangle.
The chosen contractions have mutually orthogonal initial
spaces and mutually orthogonal final spaces, so their sum
is a contraction. Their defects under $F^2-F$ lie in
orthogonal rectangular Hilbert spaces. Consequently,
\[
 \begin{aligned}
 \delta^2
 &\ge\frac{\rho^6}{32d}
       \sum_{\text{matched }(i,j)}\max(m_i,m_j)\ge\frac{\rho^6}{64d}
       \sum_{\text{matched }(i,j)}(m_i+m_j).
 \end{aligned}
\]
Thus the sum of the matched endpoint blocks has trace at most \(
 {64\delta^2}/{\rho^6}\le64\rho^2\). Remove these blocks and let $p$ be the sum of the remaining
ones. Maximality of the matching means that no bad edge joins
two retained indices. We have therefore proved
\begin{equation}\label{cp:spectral-dichotomy}
 \tau(1-p_n)\longrightarrow_\omega0,\qquad
 \operatorname{spec}(F_{ij})
 \subset[-\rho,\rho]\cup[1-\rho,1]
 \quad(i,j\text{ retained}).
\end{equation}

Call retained indices $i,j$ related, and write $i\sim j$,
when $F_{ij}$ has an eigenvalue in $[1-\rho,1]$; we'll call such an eigenvalue and the corresponding eigenvector \emph{high}.
For such an eigenvector $v$, normalized as above,
\eqref{cp:fourth-moment} gives
\[
 t\ge(1-O(\rho))M,\qquad
 \Tr|v|^4\le(1+O(\rho))M.
\]
Complete its polar part to a rank-$t$ partial isometry
$u_{ij}$. If $s_1,\ldots,s_t$ are the singular values of $v$,
including zeros, then
\[
 \begin{aligned}
 \|v-u_{ij}\|_{\mathrm{HS}}^2
 &=\sum_{r=1}^t(s_r-1)^2\le\sum_{r=1}^t(s_r^2-1)^2=\Tr|v|^4-2M+t
 =O(\rho)M.
 \end{aligned}
\]
Thus
\begin{equation}\label{cp:polar-nearfixed}
 \|v-u_{ij}\|_{\mathrm{HS}}=O(\sqrt\rho)\sqrt M.%,
 %\qquad
 %\|F_{ij}(u_{ij})-u_{ij}\|_{\mathrm{HS}}
  %     =O(\sqrt\rho)\sqrt M.
\end{equation}

To control products, choose a Kraus representation
$F(x)=\sum_\ell a_\ell^*xa_\ell$.
Unitality and trace preservation give
\[
 \sum_\ell a_\ell^*a_\ell
 =\sum_\ell a_\ell a_\ell^*=1.
\]
Expanding the commutators therefore yields
\[
 \mathcal E_F(x)
 :=\|x\|_{\mathrm{HS}}^2-\operatorname{Re}\Tr(x^*F(x))
 =\tfrac12\sum_\ell\|[a_\ell,x]\|_{\mathrm{HS}}^2.
\]
In particular, $\mathcal E_F^{1/2}$ is a seminorm.
Since $F$ is self-adjoint with spectrum in $[-1,1]$,
\[
 0\le I-F\le2I,\qquad (I-F)^2\le2(I-F).
\]
For a high eigenvector $v$ and its polar completion $u$,
\[
 \mathcal E_F(v)\le\rho M,\qquad
 \mathcal E_F(u)^{1/2}
 \le\mathcal E_F(v)^{1/2}
       +\sqrt2\|u-v\|_{\mathrm{HS}}
 =O(\sqrt\rho)\sqrt M.
\]
Also $\mathcal E_F(x^*)=\mathcal E_F(x)$. For contractions $x,y$, the commutator product rule and
Minkowski's inequality give
\[
 \mathcal E_F(xy)^{1/2}
 \le\mathcal E_F(x)^{1/2}+\mathcal E_F(y)^{1/2}.
\]
Combining this with $(I-F)^2\le2(I-F)$ proves
\begin{equation}\label{cp:product-bound}
 \|F(xy)-xy\|_{\mathrm{HS}}
 \le\sqrt2\bigl(\mathcal E_F(x)^{1/2}
                  +\mathcal E_F(y)^{1/2}\bigr).
\end{equation}

We now show that each rectangle has at most one high
eigendirection. Otherwise choose orthogonal high eigenvectors
$v,w$, both of squared Hilbert--Schmidt norm $M$,
and let $u,z$ be their rank-$t$ polar completions.
By \eqref{cp:polar-nearfixed},
\[
 |\Tr(u^*z)|
 \le|\Tr((u-v)^*z)|+|\Tr(v^*(z-w))|
 =O(\sqrt\rho)M.
\]
On the other hand, for projections $P,Q$ in an
$m$-dimensional space,
\(
 \Tr(PQ)\ge\rank P+\rank Q-m.
\)
Applying this to the rank-$t$ projections $uu^*,zz^*$
in the $i$-th corner gives
\[
 \|u^*z\|_{\mathrm{HS}}^2
 =\Tr((uu^*)(zz^*))
 \ge2t-m_i
 \ge(1-O(\rho))M.
\]
Now let $y=u^*z\in e_jM_de_j$. The product estimate gives
$\|F_{jj}(y)-y\|_{\mathrm{HS}}=O(\sqrt\rho)\sqrt M$.
Hence the scalar diagonal estimate implies
\[
 \|y-\tau_j(y)e_j\|_{\mathrm{HS}}
 \le\|y-F_{jj}(y)\|_{\mathrm{HS}}
       +\sigma\|y\|_{\mathrm{HS}}\\
 =O(\sqrt\rho)\sqrt M.
\]
Since $m_j\ge t\ge(1-O(\rho))M$, the trace estimate also gives
\[
 \|\tau_j(y)e_j\|_{\mathrm{HS}}
 =\frac{|\Tr(y)|}{\sqrt{m_j}}
 =O(\sqrt\rho)\sqrt M.
\]
Thus $\|y\|_{\mathrm{HS}}^2=O(\rho)M$, contradicting its
lower bound for sufficiently small $\rho$. Thus each rectangle has at most one high eigendirection.

For a nonzero matrix $x$, let $P_x$ denote the orthogonal
projection onto $\mathbb Cx$ in the Hilbert--Schmidt inner product.
For $i\sim j$, let $v$ be a high eigenvector of $F_{ij}$.
Uniqueness of the high eigendirection and
\eqref{cp:spectral-dichotomy} give
\(
 \|F_{ij}-P_v\|_{2\to2}\le\rho.
\)
For an unrelated pair, the same dichotomy gives
\(
 \|F_{ij}\|_{2\to2}\le\rho.
\)
On diagonal blocks take $u_{ii}=e_i$.

The relation $\sim$ is reflexive because $F(e_i)=e_i$,
and symmetric because $F$ preserves adjoints.
For transitivity, suppose $i\sim j\sim k$ and put
\[
 w=u_{ij}u_{jk},\qquad M_0=\max(m_i,m_j,m_k).
\]
The dimension comparisons for the two related pairs imply
that all three dimensions equal $(1-O(\rho))M_0$.
Using the projection trace inequality in the $j$-th corner,
we obtain
\[
 \|w\|_{\mathrm{HS}}^2
 =\Tr\bigl((u_{ij}^*u_{ij})(u_{jk}u_{jk}^*)\bigr)\\
 \ge\min(m_i,m_j)+\min(m_j,m_k)-m_j\\
 \ge(1-O(\rho))M_0.
\]
At the same time, \eqref{cp:product-bound} gives
\(
 \|F_{ik}(w)-w\|_{\mathrm{HS}}
 =O(\sqrt\rho)\sqrt{M_0}.
\)
If $i\not\sim k$, then \eqref{cp:spectral-dichotomy} gives
$\|F_{ik}\|_{2\to2}\le\rho$, and hence
\(
 \|F_{ik}(w)-w\|_{\mathrm{HS}}
 \ge(1-\rho)\|w\|_{\mathrm{HS}},
\)
a contradiction for sufficiently small $\rho$.
Thus $\sim$ is an equivalence relation.

\smallskip\noindent
\textbf{Matrix units and the whole algebra.}
In each equivalence class $K$, choose a least-rank block $o$,
of rank $t_K$, and take the isometries
$
 U_i=u_{io}:\mathbb C^{t_K}\longrightarrow\mathbb C^{m_i}$ and
 $U_o=I.$
Put
\[
 q_i=U_iU_i^*,\qquad W_{ij}=U_iU_j^*\quad(i,j\in K).
\]
Their ranges lie in mutually orthogonal blocks, so
\[
 W_{ij}^*=W_{ji},\qquad
 W_{ij}W_{k\ell}=\delta_{jk}W_{i\ell},\qquad
 W_{ii}=q_i.
\]
Products between different classes are zero.
These are therefore exact matrix units.
Since every $i\in K$ is directly related to $o$,
\(
 t_K/m_i=1-O(\rho)
\)
uniformly, independently of the size of the class.

By \eqref{cp:product-bound},
\(
 \|F_{ij}(W_{ij})-W_{ij}\|_{\mathrm{HS}}
 =O(\sqrt\rho)\sqrt{t_K}.
\)
For $i\ne j$ in the same class, set
$w=W_{ij}/\sqrt{t_K}$, so that $\|w\|_{\mathrm{HS}}=1$,
and let $v$ span the unique high eigendirection of $F_{ij}$.
On its orthogonal complement, \eqref{cp:spectral-dichotomy}
places the spectrum of $F_{ij}$ in $[-\rho,\rho]$. Thus
\[
 \|(I-P_v)w\|_{\mathrm{HS}}
 \le\frac{\|(I-F_{ij})w\|_{\mathrm{HS}}}{1-\rho}
 =O(\sqrt\rho).
\]
For a unit vector spanning one line, its distance from a
second line equals the operator-norm distance between their
rank-one orthogonal projections. Consequently,
\[
 \begin{aligned}
 \|F_{ij}-P_{W_{ij}}\|_{2\to2}
 &\le\|F_{ij}-P_v\|_{2\to2}
       +\|P_v-P_{W_{ij}}\|_{2\to2}\le\rho+\|(I-P_v)w\|_{\mathrm{HS}}
 =O(\sqrt\rho).
 \end{aligned}
\]
On the diagonal,
\[
 \|F_{ii}-P_{q_i}\|_{2\to2}
 \le\sigma+\|P_{e_i}-P_{q_i}\|_{2\to2}
 =\sigma+\sqrt{1-t_K/m_i}
 =O(\sqrt\rho).
\]
Let $q=\sum_iq_i$, summing over retained indices, and let
\[
 A^0=\bigoplus_K(M_{|K|}\otimes I_{t_K})\subset qM_dq
\]
be the algebra generated by the matrix units.
Define $G(x)=E_{A^0}(qxq)$.
The matrix units are orthogonal in Hilbert--Schmidt inner
product. Thus the restriction of $G$ to $e_iM_de_j$ is
$P_{W_{ij}}$ within a class, and zero between classes.
Orthogonality of the rectangular Hilbert spaces now gives
\[
 \|F|_{pM_dp}-G|_{pM_dp}\|_{2\to2}=O(\sqrt\rho).
\]
The bound is independent of the number of blocks.
Also,
\[
 \tau(p-q)
 =\frac1d\sum_i(m_i-t_{K(i)})
 \le O(\rho)\frac1d\sum_i m_i
 \le O(\rho).
\]
Restoring the coordinate index,
\[
 \tau(1-q_n)
 \le\tau(1-p_n)+O(\rho_n)\longrightarrow_\omega0.
\]

Put
$
 A_n=A_n^0\oplus(1-q_n)M_{d_n}(1-q_n).
$
This algebra contains $D_n$: for a retained index,
$e_{i,n}=q_{i,n}+(e_{i,n}-q_{i,n})$ belongs to $A_n$,
and every removed block lies under $1-q_n$.
Its conditional expectation is
$
 E_{A_n}(x)=G_n(x)+(1-q_n)x(1-q_n).
$
Since $q_n\le p_n$, we have $G_n(x)=G_n(p_nxp_n)$.
Therefore
\[
F_n(x)-E_{A_n}(x)
 =F_n(x-p_nxp_n)
       +(F_n-G_n)(p_nxp_n)\\
 -(1-q_n)x(1-q_n).
\]
For a contraction $x$,
\[
 \|x-p_nxp_n\|_2^2
 =\|(1-p_n)x\|_2^2+\|p_nx(1-p_n)\|_2^2\\
 \le2\tau(1-p_n),
\]
and
$\|(1-q_n)x(1-q_n)\|_2\le\sqrt{\tau(1-q_n)}$.
Contractivity of $F_n$ thus gives
\[
 \|F_n-E_{A_n}\|_{\infty\to2}
 \le O(\sqrt{\rho_n})+\sqrt{2\tau(1-p_n)}
                       +\sqrt{\tau(1-q_n)}
 \longrightarrow_\omega0.
\]
The choices $A_n=M_{d_n}$ on the coordinates omitted at the
start do not affect this limit.
Coordinate expectations induce the expectation onto
$[A_n]_\omega$. Their induced map is $E_N$, so
$N=[A_n]_\omega$.

Finally, coordinatewise commutation gives
$[Z(A_n)]_\omega\subset Z(N)$.
For the reverse inclusion, take $z\in Z(N)$ and choose
uniformly bounded representatives $z_n\in A_n$.
Then
\[
 \sup_{u\in U(A_n)}\|[z_n,u]\|_2\longrightarrow_\omega0.
\]
Indeed, choose a maximizing unitary $u_n$ in each coordinate,
which is possible by compactness. Since
$(u_n)_\omega\in U(N)$, centrality of $z$ makes the ultralimit
of these maxima zero.

Let $\mu_n$ be normalized Haar measure on $U(A_n)$ and put
\[
 z_n^0=\int_{U(A_n)}u z_nu^*\,d\mu_n(u).
\]
Then $z_n^0\in Z(A_n)$, $\|z_n^0\|\le\|z_n\|$, and
\[
 \|z_n-z_n^0\|_2
 \le\sup_{u\in U(A_n)}\|[z_n,u]\|_2
 \longrightarrow_\omega0.
\]
Thus $z=(z_n^0)_\omega\in[Z(A_n)]_\omega$, proving
$Z(N)=[Z(A_n)]_\omega$.
\end{proof}

\section{Deducing the main results}
\label{fin:section}
The remaining main results now follow directly. The lifts of
Corollary~\ref{phy:lifts} satisfy the hypotheses of
Theorem~\ref{cp:rounding}. This proves Theorem~\ref{main:internality}.
In particular, this gives a positive solution to the centralizer problem
as formulated in~\cite{AlekseevThomCentralizers,ThomConditional}.

If $H<G$ is infranormal but not normal and both groups are
Kazhdan, then~\cite[Theorem~1.3]{ThomConditional} shows that
$G\ast_H G$ is not hyperlinear. This proves
Corollary~\ref{main:nonhyperlinear}. A concrete such pair is given by
\[
 H=\mathrm{EL}_3(\mathbb F_2[x_1,x_2,x_3]),\qquad
 G=\mathrm{EL}_3(\mathbb F_2[x_1^{\pm1},x_2^{\pm1},x_3^{\pm1}])
                      \rtimes\mathrm{SL}_3(\mathbb Z),
\]
where $\mathrm{SL}_3(\mathbb Z)$ acts by monomial substitutions; see
\cite[Theorem~E]{KunThomNonsofic}.

Finally, suppose that $H<G$ satisfies the assumptions of
Corollary~\ref{main:centralizer-normal} and that $G$ is hyperlinear.
Choose an embedding $\pi:G\to U(\M)$ into a tracial matrix
ultraproduct. By~\cite[Theorem~1.2]{ThomConditional}, the relative
commutant $\pi(H)'\cap\M$ is normalized by $\pi(G)$. Since $\pi$ is
injective,
\[
 \pi(C_G(H))=\pi(G)\cap\bigl(\pi(H)'\cap\M\bigr).
\]
The right-hand side is normal in $\pi(G)$, so $C_G(H)$ is normal in
$G$. This proves Corollary~\ref{main:centralizer-normal}.

\section*{Acknowledgments}

The second author was partially supported by the National Key R\&D
Program of China (Grant No.~2024YFA1014400). He thanks Leheng Chen,
Guoxiong Gao, Jiedong Jiang, Haocheng Ju, Shurui Liu, Zeming Sun,
Yuefeng Wang, Bin Wu, Liang Xiao, and Bin Dong from the Rethlas/Danus
team, as well as Ruochuan Liu and Gang Tian. He also thanks Mikael de
la Salle for comments on~\cite{LiuIdeas}.

The first and third authors thank Jesse Peterson for sharing his notes on spectral-gap
lifting and for emphasizing the difficulties of controlling
off-diagonal terms and preserving expansion during projection
improvement. They also thank Lewis Bowen and Konrad Wr\'obel for sharing
their work on noncommutative analogues of Kun's theorem. 
Finally, they thank Francesco Fournier-Facio for bringing the
manuscript~\cite{LiuIdeas} to their attention.

\end{document}